\documentclass[a4paper,10pt]{article}

\usepackage{amssymb, amsmath, mathtools, amsfonts, times, amsthm, amscd, setspace, leftidx, enumerate, xfrac}
\usepackage{tikz-cd}
\usepackage{pdflscape}
\usepackage{adjustbox}
\input xy 
\xyoption{all}

\usepackage[titletoc,title]{appendix}
\usepackage{graphicx}
\graphicspath{ {Graphics/} }
\usepackage{capt-of}

\theoremstyle{plain}

\newtheorem{proposition}{Proposition}[section]
\newtheorem{defi}[proposition]{Definition}
\newtheorem{lemma}[proposition]{Lemma}
\newtheorem{thm}[proposition]{Theorem}
\newtheorem{cor}[proposition]{Corollary}
\newtheorem{rmk}[proposition]{Remark}

\newtheorem{ex}[proposition]{Example}

\newcommand{\ct}[1]{\mathcal{#1}}
\newcommand{\ov}[1]{\overline{#1}}

\newcommand{\pr}[1]{\prescript{\vee}{}{#1}}
\newcommand{\un}[0]{\mathtt{1}}
\newcommand{\field}[0]{\mathbb{K}}
\newcommand{\bim}[0]{\prescript{}{A}{\mathcal{M}}_{A}}

\newcommand{\tn}[0]{\otimes}
\newcommand{\cvl}[0]{\mathrm{coev}}
\newcommand{\evl}[0]{\mathrm{ev}}
\newcommand{\cvr}[0]{\underline{\mathrm{coev}}}
\newcommand{\evr}[0]{\underline{\mathrm{ev}}}
\newcommand{\omx}[3]{\prescript{}{#1}{#2}_{#3}}
\newcommand{\op}[0]{\mathrm{op}}
\newcommand{\id}[0]{\mathrm{id}_{\ct{C}}}
\allowdisplaybreaks
\begin{document}

\title{Pivotal Objects in Monoidal Categories and Their Hopf Monads}
\author{Aryan Ghobadi \\ \small{Queen Mary University of London }\\\small{ School of Mathematics, Mile End Road}\\\small{ London E1 4NS, UK }\\ \small{Email: a.ghobadi@qmul.ac.uk}}
\date{}

\maketitle
\begin{abstract}
An object $P$ in a monoidal category $\ct{C}$ is called pivotal if its left dual and right dual objects are isomorphic. Given such an object and a choice of dual $Q$, we construct the category $\ct{C}(P,Q)$, of objects which intertwine with $P$ and $Q$ in a compatible manner. We show that this category lifts the monoidal structure of $\ct{C}$ and the closed structure of $\ct{C}$, when $\ct{C}$ is closed. If $\ct{C}$ has suitable colimits we show that $\ct{C}(P,Q)$ is monadic and thereby construct a family of Hopf monads on arbitrary closed monoidal categories $\ct{C}$. We also introduce the pivotal cover of a monoidal category and extend our work to arbitrary pivotal diagrams. 
\end{abstract}
\begin{footnotesize}2010\textit{ Mathematics Subject Classification}: 18D10, 18D15, 16T99, 18C20
\\\textit{Keywords}: Monoidal category, closed category, pivotal category, Hopf monad, Hopf algebra, tensor category\end{footnotesize}
\section{Introduction}
Hopf monads were originally introduced as generalisations of Hopf algebras in braided monoidal categories, from the setting of braided categories to arbitrary monoidal categories. In \cite{moerdijk2002monads}, Hopf monads were defined as monads which lift the monoidal structure of a monoidal category to their category of modules. These monads are now referred to as \emph{bimonads} or \emph{opmonoidal monads}, whereas monads which lift the closed structure of a closed monoidal category as well as its tensor, are called Hopf monads, according to \cite{bruguieres2011hopf}. Initially, Hopf monads were defined for rigid monoidal categories \cite{bruguieres2007hopf} and many of the usual results about Hopf algebras were extended to this setting. Additionally, Hopf monads have proved of particular importance in the study of tensor categories \cite{bruguieres2011exact} and topological field theories \cite{turaev2017monoidal}. However, in both settings the categories in consideration are rigid. Although in \cite{bruguieres2011hopf}, the theory of Hopf monads was extended to arbitrary monoidal categories, many of the essential theorems discussed in \cite{bruguieres2007hopf} have not been extended to the general setting. We believe this is partly due to the lack of examples which have been studied, when the category is not rigid. Inspired from our work on bimodule connections and Hopf algebroids in \cite{ghobadi2020hopf}, under the setting of noncommutative differential geometry \cite{beggs2019quantum}, in the present work, we study objects in arbitrary monoidal categories, which have isomorphic left and right duals. We call these objects pivotal and when provided with such a pivotal pair, $P$ and $Q$, in a monoidal category $\ct{C}$, we construct the category of $P$ and $Q$ intertwined objects $\ct{C}(P,Q)$. The constructed category lifts the monoidal structure of $\ct{C}$ and the closed structure of $\ct{C}$, when $\ct{C}$ is closed. Consequently, if $\ct{C}$ is closed and has suitable colimits, we construct a Hopf monad corresponding to such a pivotal pair so that its Eilenberg-Moore category recovers $\ct{C}(P,Q)$. The merit of this construction is that it does not require the category to be braided or even to have a non-trivial center. We later show that $\ct{C}(P,Q)$ can be constructed as the dual of a certain monoidal functor and our monad can also be recovered from the Tannaka-Krein reconstruction of bimonads developed in the recent work \cite{shimizu2019tannaka}. 

In \cite{ghobadi2020hopf}, we construct a family of Hopf algebroids corresponding to first order differential caluli. The modules of the Hopf algebroid recover a closed monoidal subcategory of bimodule connections, however, the differential calculus must be pivotal for the construction to work. We explain this ingredient in a much more general setting, here. If $\ct{C}$ is a monoidal category and $P$ an object in $\ct{C}$, one can construct a category of $P$-intertwined objects, whose objects are pairs $(A,\sigma)$, where $A$ is an object of $\ct{C}$ and $\sigma :A\tn P\rightarrow P\tn A$ an invertible morphism. This category naturally lifts the monoidal structure of $\ct{C}$, in a similair fashion to the monoidal structure of the center of $\ct{C}$. The key feature in our construction is the following: if $P$ is pivotal and we choose a dual, $Q$, of $P$, then any invertible morphism $\sigma$ induces two $Q$-intertwinings on the object $A$, namely \ref{Eqovsig1} and \ref{Eqovsig2}. In order to obtain a closed monoidal category, we must restrict to the subcategory of pairs where these induced $Q$-intertwinings are inverse. We denote this category, corresponding to the pivotal pair $P$ and $Q$, by $\ct{C}(P,Q)$ and describe its monoidal structure in Theorem \ref{TMon}. Our main results are Theorem \ref{TCld} and Corollary \ref{CRCld}, which show that $\ct{C}(P,Q)$ lifts left and right closed structures on $\ct{C}$, when they exist. We also discuss the construction in the cases where $\ct{C}$ is rigid, Corollary \ref{CRig}, and when $\ct{C}$ has a pivotal structure which is compatible with $P$ and $Q$, Theorem \ref{TpivCP}.

Pivotal categories were introduced in \cite{barrett1999spherical} and their study is vital for topological field theories, \cite{turaev2017monoidal}. However, a study of individual objects in a monoidal category which have isomorphic left and right duals has not been produced. In \cite{shimizu2015pivotal}, the pivotal cover of a rigid monoidal category was introduced, in connection with Frobenius-Schur indicators discussed in \cite{ng2007frobenius}. We introduce the pivotal cover $\ct{C}^{piv}$ of an arbitrary monoidal category $\ct{C}$, in Definition \ref{DPivMo}, from a different point of view which arose in \cite{ghobadi2020hopf}, namely pivotal morphisms. The pivotal cover of a monoidal category has pivotal pairs as objects, and suitable pivotal morphisms between them, so that any strong monoidal functor from a pivotal category to the original category, factors through the pivotal cover, Theorem \ref{TCPiv}. The construction in \cite{shimizu2015pivotal} requires all objects to have left duals and a choice of distinguished left dual for each object i.e. for the category to be left rigid, while our construction avoids these issue by taking pivotal pairs as objects of $\ct{C}^{piv}$. 

The applications of our work are spread as examples throughout the article. In Section \ref{SPivOb}, we observe that dualizable objects in braided categories and ambidextrous adjunctions are simply examples of pivotal objects in monoidal categories. In Section 4.2 of \cite{ghobadi2020hopf}, we have presented several other examples, in the format of differential calculi, where the space of 1-forms is a pivotal object in the monoidal category of bimodules over the algebra of noncommutative functions. We explain this setting briefly in Example \ref{Ebim}. The Hopf monad constructed in this case, becomes a Hopf algebroid, Example \ref{EHpfAlebroid}, which is a subalgebra of the Hopf algebroid of differential operators defined in \cite{ghobadi2020hopf}. Additionally, to construct the sheaf of differential operators in the setting of \cite{ghobadi2020hopf}, we require the wedge product between the space of 1-forms and 2-forms to be a pivotal morphism. A direct consequence of Example \ref{EPivMo} is that any bicovariant calculus over a Hopf algebra, satisfies this condition. 

As alluded to in Example \ref{EHopfAlg}, the Hopf monads constructed here are in some sense a noncommutative version of the classical Hopf algebra $\ct{O}(GL(n))$, where instead of $n\times n$ matricies, the monad provides matrix-like actions of pivotal pairs. In Remark \ref{RNGL}, we note that for a finite dimensional vector space, the resulting Hopf algebra becomes precisely a quotient of the free matrix Hopf algebra, $\ct{NGL}(n)$, discussed in \cite{skoda2003localizations}. In fact, any invertible $n\times n$-matrix provides us with a pivotal pairing of an $n$-dimensional vectorspace and thereby a Hopf algebra, Example \ref{EHopfMany}. More generally, in Theorem \ref{ThmTAug}, we show that the Hopf monad constructed is augmented if and only if the pair $P$ and $Q$ is a pivotal pair in the center of the monoidal category. Consequently, using the theory of augmented Hopf monads from \cite{bruguieres2011hopf}, we a construct the braided Hopf algebra corresponding to every pivotal pair in the center of the monoidal category. 

Finally, we must point out how our work can be interpreted in terms of \cite{shimizu2019tannaka}. Duals of monoidal functors were defined in \cite{majid1992braided}, by Majid, as a generalisation of the center of a monoidal category. Tannaka-Krein reconstruction for Hopf monads, as described in \cite{shimizu2019tannaka}, takes the data of a strong monoidal functor and produces a Hopf monad whose module category, recovers the dual of the monoidal functor. In Section \ref{SDual}, we briefly review these topics and show that pivotal pairs in a monoidal category correspond to strict monoidal functors from the smallest pivotal category, which we call $\mathrm{Piv}(1)$, into the category. From this point of view, one can recover the same Hopf monad structure from the approach of Tannaka-Krein reconstruction. We also provide an additional result concerning the pivotal structure of the dual monoidal category.   

Lastly, we would like to remark that the Hopf monads constructed here should be the simplest examples which generalise the theory of Hopf algebroids with bijective antipodes, described in \cite{bohm2004hopf}, to the monadic setting. Although antipodes for Hopf monads have been discussed in both the rigid setting \cite{bruguieres2007hopf} and the general setting \cite{bohm2016hopf}, neither cover the case of Hopf algebroids over noncommutative bases which admit bijective antipodes. When restricted to the category of bimodules over an arbitrary algebra, our Hopf monads correspond to Hopf algebroids which in fact admit involutory antipodes, Example \ref{EHpfAlebroid}.
 
\textbf{Organisation:} In Section \ref{SPre}, we review the theory of Hopf monads and braided Hopf algebras and the necessary background on duals and closed structures in monoidal categories. In Section \ref{SPivOb}, we introduce the notion of pivotal objects and morphisms in an arbitrary monoidal categories and introduce the pivotal cover. In Section \ref{SC(P)}, we construct $\ct{C}(P,Q)$ and review some of its properties and in Section \ref{SMnd} we construct its corresponding Hopf monad. In Section \ref{SExtension}, we briefly discuss the generalisation of our work to arbitrary pivotal diagrams and in Section \ref{SDual} we provide an alternative description of $\ct{C}(P,Q)$ as the dual of a monoidal functor. The proof of Theorem \ref{TCld} requires several large commutative diagrams which are presented in Section \ref{SDiag}, at the end of our work. 
 
\section{Preliminaries}\label{SPre}
We assume basic categorical knowledge and briefly recall the theory of monads and monoidal categories from \cite{mac2013categories,turaev2017monoidal} and the theory of Hopf monad and bimonads from \cite{bruguieres2011hopf,bruguieres2007hopf}.
\subsection{Monads}
A \emph{monad} $T$ on a category $\mathcal{C}$, consists of a triple $(T,\mu , \eta)$, where $T:\ct{C}\rightarrow \ct{C}$ is an endofunctor with natural transformations $\mu :TT\rightarrow T$ and $\eta : \mathrm{id}_\mathcal{C} \rightarrow T$ satisfying satisfying $\mu (T\mu )=\mu \mu_{T}$ and $ \mu T\eta = \mathrm{id}_{T} = \mu \eta_{T} $. Any monad gives rise to an adjunction $F_{T}\dashv U_{T}: \mathcal{C}_{T}\leftrightarrows \mathcal{C}$, where $\mathcal{C}_{T}$ is the \emph{Eilenberg-Moore category} associated to $T$. The category $\mathcal{C}_{T}$ consists of pairs $(X,r)$, where $X$ is an object of $\mathcal{C}$ with a $T$\emph{-action} $r:TX\rightarrow X$, satisfying $r\mu_{X}=r(Tr)$ and $r\eta=\mathrm{id}_{X}$ and morphims which commute with such $T$-actions. The \emph{free} functor is defined by $F_{T}(X)=(TX,\mu_{X})$ and the \emph{forgetful} functor by $U_{T}(X,r)=X$. Conversely, any adjunction $F\dashv G: \mathcal{D}\leftrightarrows \mathcal{C}$ gives rise to a monad via its unit $\eta : \mathrm{id}_{\mathcal{C}}\rightarrow GF$ and counit $\epsilon :FG\rightarrow \mathrm{id}_{\mathcal{D}}$. The triple produced is $(GF, G\epsilon_{F},\eta)$. Hence, there is a natural functor $K=F_{T}G:\mathcal{D}\rightarrow \mathcal{C}_{T}$ called the \emph{comparison functor}. We say functor $G$ is \emph{monadic} if $K$ is an equivalence of categories. For more detail on monads, we refer the reader to Chapter VI of \cite{mac2013categories}, since we will only present Beck's Theorem and later utilise it. 
\begin{thm}\label{TBeck}[Beck's Theorem] Given an adjunction $F\dashv G: \mathcal{D}\leftrightarrows \mathcal{C}$, $G$ is monadic if and only if the functor $G$ creates coequalizers for parallel pairs $f,g: X\rightrightarrows Y$ for which $Gf, Gg$ has a split coequalizer.
\end{thm}
\subsection{Monoidal Categories}
We call $(\mathcal{C}, \otimes, 1_{\otimes},\alpha,l,r)$ a monoidal category, where $\mathcal{C}$ is a category, $1_{\otimes}$ an object of $\mathcal{C}$, $\otimes:\mathcal{C}\times \mathcal{C}\rightarrow\mathcal{C}$ a bifunctor and $\alpha: (\mathrm{id}_{\mathcal{C}}\otimes  \mathrm{id}_{\mathcal{C}})\otimes \mathrm{id}_{\mathcal{C}}\rightarrow \mathrm{id}_{\mathcal{C}}\otimes(\mathrm{id}_{\mathcal{C}}\otimes \mathrm{id}_{\mathcal{C}}), l:1_{\otimes}\otimes \mathrm{id}_{\mathcal{C}}\rightarrow \mathrm{id}_{\mathcal{C}}$ and $r:\mathrm{id}_{\mathcal{C}}\otimes 1_{\otimes} \rightarrow \mathrm{id}_{\mathcal{C}}$ natural isomorphisms satisfying coherence axioms, as presented in Section 1.2 of \cite{turaev2017monoidal}. We call $\un_{\tn}$ the \emph{monoidal unit}.

In what follows, we assume that all monoidal categories in question are \emph{strict} i.e. $\alpha, l$ and $r$ are all identity morphisms. Additionally, there exists a corresponding monoidal structure on $\ct{C}$, the \emph{opposite monoidal category}, which we denote by $(\ct{C},\tn^{op})$ and is defined by composing $\tn$ with the flip functor i.e. $X\tn^{\op} Y=Y\tn X$ for pairs of objects $X,Y$ of $\ct{C}$. This notion should not be confused with the notion of the opposite category, where morphisms are reversed and $\ct{C}^{\op}$ will only refer to the opposite monoidal category in this work.

A functor $F:\mathcal{C}\rightarrow\mathcal{D}$ between monoidal categories is said to be \emph{(strong) monoidal} if the exists a natural (isomorphism) transformation $F_{2}(-,-): F(-)\otimes_{\mathcal{D}}F(-)\rightarrow F(-\otimes_{\mathcal{C}}-)$ and a (isomorphism) morphism $F_{0} : 1_{\otimes}\rightarrow F(1_{\otimes})$ satisfying 
\begin{align*}
F_{2}(X\otimes Y,Z)( F_{2}(X,Y)\otimes \mathrm{id}_{F(Z)})=F_{2}&(X, Y\otimes Z)(\mathrm{id}_{F(X)}\otimes  F_{2}(Y,Z))
\\F_{2}(X,1_{\otimes})(\mathrm{id}_{F(X)}\otimes F_{0})=\mathrm{id}_{F(X)}&=F_{2}(1_{\otimes},X)(F_{0}\otimes \mathrm{id}_{F(X)})
\end{align*}
where we have omitted the subscripts denoting the ambient categories, since they are clear from context. A functor is said to be \emph{opmonoidal} or \emph{comonoidal} if all morphisms in the above definition are reversed. A strong monoidal fucntor $F$ is called \emph{strict} monoidal if the natural isomorphisms $F_{2}$ and $F_{0}$ are identity morphisms. A monoidal category is said to be \emph{braided} if there exists a natural isomorphism $\Psi_{X,Y}: X\tn Y\rightarrow Y\tn X $ satisfying braiding axioms described in Section 3.1 of \cite{turaev2017monoidal}. 

\textbf{Notation.} We will abuse notation and write $X$ instead of the morphism $\mathrm{id}_{X}$ whenever it is feasible. We will also omit $\tn$ when writing long compositions of morphisms i.e. $AfB$ will denote the morphisms $\mathrm{id}_{A}\tn f \tn \mathrm{id}_{B}$ for arbitrary objects $A $ and $B$ and morphism $f$ in $\ct{C}$.

The \emph{(lax) center} of a monoidal category $(\mathcal{C},\otimes,1_{\otimes})$ has pairs $(X,\tau)$ as objects, where $X$ is an object in $\mathcal{C}$ and $\tau : X\otimes -\rightarrow -\otimes X$ is a natural (transformation) isomorphism satisfying $\tau_{1_{\otimes}}= \mathrm{id}_{X}$ and $(\mathrm{id}_{M}\otimes \tau_{N})(\tau_{M}\otimes \mathrm{id}_{N})=\tau_{M\otimes N}$, and morphisms $f: X\rightarrow Y$ of $\mathcal{C}$, satisfying $(\mathrm{id}_{\mathcal{C}}\otimes f)\tau =\nu (f\otimes \mathrm{id}_{\mathcal{C}})$, as morphism $f:(X,\tau)\rightarrow (Y,\nu )$. We denote the lax center and center by $Z^{lax}(\mathcal{C})$ and $Z(\mathcal{C})$, respectively. The center is often referred to as the \emph{Drinfeld-Majid center} and the lax center is sometimes referred to as the \emph{prebraided} or \emph{weak} center. The (lax) center has a monoidal structure via 
$$ (X,\tau )\otimes (Y, \nu):= (X\otimes Y , (\tau \otimes \mathrm{id}_{Y})(\mathrm{id}_{X}\otimes\nu ) )$$ 
and $(1_{\otimes}, \mathrm{id}_{\ct{C}})$ acting as the monoidal unit, so that the forgetful functor to $\mathcal{C}$ is strict monoidal. The center $Z(\ct{C})$ is also braided by $\Psi_{(X,\tau ), (Y, \nu)} = \tau_{Y}$.

\subsection{Rigid and Closed Monoidal Categories}\label{SRigClos}
For any object, $X$, in a monoidal category $\mathcal{C}$, we say an object $\pr{X}$ is a \emph{left dual} of $X$, if there exist morphisms $\mathrm{ev}_{X}:\pr{X}\otimes X\rightarrow \un_{\otimes}$ and $\mathrm{coev}_{X}:\un_{\otimes}\rightarrow X\otimes \pr{X}$ such that 
$$(\mathrm{ev}_{X}\otimes \mathrm{id}_{\pr{X}})(\mathrm{id}_{\pr{X}}\otimes \mathrm{coev}_{X})=\mathrm{id}_{\pr{X}} , \quad (\mathrm{id}_{X}\otimes \mathrm{ev}_{X} )(\mathrm{coev}_{X}\otimes \mathrm{id}_{X})=\mathrm{id}_X$$ 
In such a case, we call $X$ a \emph{right dual} for $\pr{X}$. Furthermore, a right dual of an object $X$ is denoted by $X^{\vee}$, with evalutation and coevaluation maps denoted by $\underline{\mathrm{ev}}_{X}:X\otimes X^{\vee}\rightarrow \un_{\otimes}$ and $\underline{\mathrm{coev}}_{X}:\un_{\otimes}\rightarrow X^{\vee}\otimes X$, respectively. We will refer to evaluation and coevaluation maps as such, as \emph{duality morphisms}. We say an object $X$ is \emph{dualizable} if has both a left dual and a right dual. The category $\mathcal{C}$ is said to be left (right) \emph{rigid} or \emph{autonomous} if all objects have left (right) duals. If a category is both left and right rigid, we simply call it \emph{rigid}. Usually, when a category is said to be left (or right rigid), it is assumed that we have \emph{chosen} a left dual for all objects and $\pr{X}$ denotes this specific choice of left dual for any object $X$. Given these choices, we have a contravariant functor $\pr{(-)}:\ct{C}\rightarrow \ct{C}$ which sends objects $X$ to their left duals $\pr{X}$ and morphisms $f:X\rightarrow Y$ to morphisms $(\evl_{Y} \tn \mathrm{id}_{\pr{X}})(\mathrm{id}_{\pr{Y}}\tn f\tn \mathrm{id}_{\pr{X}} ) ( \mathrm{id}_{\pr{Y}}\tn \cvl_{X})$. Similarly, $(-)^{\vee}:\ct{C}\rightarrow \ct{C}$ defines a contravariant functor on a right rigid category. 

We call a category $\mathcal{C}$ \emph{left (right) closed} if for any object $X$ there exists an endofunctor $[X,-]^{l}$ (resp. $[X,-]^{r}$) on $\mathcal{C}$ which is right adjoint to $-\otimes X$ (resp. $X\otimes -$). By definition $[-,-]^{l},[-,-]^{r}:\mathcal{C}^{op}\times\mathcal{C}\rightarrow \mathcal{C}$ are bifunctors, which we refer to as \emph{inner homs}. If a category is left and right closed, we call it \emph{closed}. Observe that if $X$ has a left (right) dual $\pr{X}$ (resp. $X^{\vee}$), the functor $-\otimes \pr{X}$ (resp. $X^{\vee}\otimes -$) is right adjoint to $- \otimes X$ (resp. $X\otimes -$) and $\pr{X}$ (resp. $X^{\vee}$) is unique up to isomorphism. Furthermore,  if $X$ has a left (right) dual, $\pr{X}\cong [X,1_{\otimes}]^{l}$ (resp. $X^{\vee}\cong [X,1_{\otimes}]^{r}$). We have adopted the notation of \cite{bruguieres2011hopf} here, and what we refer to as a left closed structure is sometimes referred to as a right closed structure in other sources.

It is well known that strong monoidal functors preserve dual objects i.e. $F(\pr{X})\cong \pr{F(X)}$ with $F_{0}F(\mathrm{ev})F_{2}(\pr{X},X)$ and $ F_{2}^{-1}(X,\pr{X}) F(\mathrm{coev})F_{0}^{-1}$ acting as the evaluation and coevaluation morphisms for $F(\pr{X})$. For left (right) closed monoidal categories $\ct{C}$ and $\ct{D}$, we say a monoidal functor $F:\ct{C}\rightarrow \ct{D}$ is left (right) closed if the canonically induced morphism $F[X,Y]^{l (r)}_{\ct{C}}\rightarrow [F(X),F(Y)]^{l (r)}_{\ct{D}}$ is an isomorphism for any pair of objects $X,Y$ in $\ct{C}$. In what follows, we will describe the inner homs so that these isomorphisms become the identity morphisms, thereby making the functor in question closed, in a clearer manner.

\subsection{Bimonads and Hopf Monads}

A monad $(T,\mu , \eta )$ on $\mathcal{C}$ is said to be a \emph{bimonad} or an \emph{opmonoidal monad} if it also has a compatible comonoidal structure $(T_{2},T_{0})$ satisfying 
$$T_{2} (X,Y) \mu_{X\otimes Y}=(\mu_{X}\otimes\mu_{Y}) T_{2}(TX,TY)T\left(T_{2}(X,Y)\right)$$
$$T_{0}\mu_{\mathtt{1}}=T_{0}T(T_{0}), \quad  T_{2}(X,Y)\eta_{X\otimes Y} = \eta_{X}\otimes \eta_{Y}, \quad T_{0}\eta_{\mathtt{1}}=\mathrm{id}_{\mathtt{1}}$$ 
where $X,Y$ are objects of $\ct{C}$. A bimonad is said to be \emph{left (right) Hopf} if the left (right) \emph{fusion operators}, denoted by $H^{l}$ (resp. $H^{r}$), and defined as  
$$\xymatrix@R-25pt{H^{l}_{X,Y}:= (\mathrm{id}_{T(X)}\otimes\mu_{Y})T_{2}(X,T(Y)): T( X\otimes T(Y))\longrightarrow T(X)\otimes T(Y)
\\ H^{r}_{X,Y}:= (\mu_{X}\otimes \mathrm{id}_{T(Y)})T_{2}(T(X),Y): T( T(X)\otimes Y)\longrightarrow  T(X)\otimes T(Y) } $$
for objects $X,Y$ of $\ct{C}$, is invertible. A bimonad is called \emph{Hopf} if it is both left and right Hopf. The above conditions can be reformulated purely in terms of $F_{T}$ and $U_{T}$: given an adjunction $F\dashv G:\ct{D}\leftrightarrows \ct{C}$, the induced monad $(GF, G\epsilon_{F},\eta)$ is a bimonad if and only if $U$ is strong monoidal. In this case, the adjunction is called \emph{comonoidal}. The adjunction is called left (right) \emph{Hopf} if $U$ is a left (right) closed functor. We briefly recall the main property of these structures and refer the reader to \cite{bruguieres2011hopf} for more detail on bimonads and Hopf monads.
\begin{thm} If $T$ is a monad on a monoidal category $\mathcal{C}$, then 
\begin{enumerate}[(I)]
\item \cite{moerdijk2002monads} Bimonad structures on $T$ are in correspondence with liftings of the monoidal structure of $\mathcal{C}$ onto $\mathcal{C}_{T}$ i.e. monoidal structures on $\mathcal{C}_{T}$ such that $U_{T}$ is strong monoidal. 
\item \cite{bruguieres2011hopf},\cite{bruguieres2007hopf} If $C$ is left (right) rigid and $T$ a bimonad, then $T$ is left (right) Hopf if and only if $C_{T}$ is left (right) rigid.
\item \cite{bruguieres2011hopf} If $\mathcal{C}$ is left (right) closed, then $T$ being left (right) Hopf is equivalent to $\mathcal{C}_{T}$ being left (right) closed and $U_{T}$ a left (right) closed functor. 
\end{enumerate}
\end{thm}
\subsection{Braided Hopf Algebras and Augmented Hopf Monads}
An \emph{algebra} or \emph{monoid} in a monoidal category $\ct{C}$ consists of a triple $(M,\mu ,\eta)$, where $M$ is an object of $\ct{C}$ and $\mu : M\otimes M\rightarrow M$ and $\eta :1_{\otimes} \rightarrow M$ are morphisms in $\ct{C}$ satisfying $\mu (\mathrm{id}_{M}\otimes \eta ) =\mathrm{id}_{M}= \mu (\eta\otimes \mathrm{id}_{M} )$ and $\mu (\mathrm{id}_{M}\otimes\mu )= \mu ( \mu\otimes \mathrm{id}_{M})$. A \emph{coalgebra} or \emph{comonoid} in $\ct{C}$ can be defined by simply reversing the morphisms in the definition of a monoid. Observe that monoid structures on an object $A$ in a monoidal category $\mathcal{C}$, correspond directly to monad structures on the endofunctor $A\otimes -$. The Eilenberg-Moore category $\mathcal{C}_{T}$ in this case is the category of left modules over $A$ i.e. objects $X$ with an action $r: A\otimes X\rightarrow  X$. 

A \emph{central bialgebra} in $\mathcal{C}$ consists of an object $(B,\tau)$ in $Z(\ct{C})$, and morphisms $m, \eta, \Delta, \epsilon$ such that $((B,\tau),m,\eta)$ is a monoid in $Z(\ct{C})$, $((B,\tau),\Delta,\epsilon)$ is a comonoid in $Z(\ct{C})$ and $(m\otimes m)(B\otimes\tau_{B} \otimes B)(\Delta\otimes\Delta)=\Delta m $ and $ \epsilon m = (\epsilon\otimes\epsilon) $ hold. Hence, the monad $T=B\tn -$ has a bimonad structure, with $T_{2}=(B\tn \tau \tn \mathrm{id}_{\ct{C}})( \Delta\tn \mathrm{id}_{\ct{C}}\tn\mathrm{id}_{\ct{C}})$ and $T_{0}=\epsilon$. A central bialgebra is called a \emph{central Hopf algebra} if there exists a morphism $S:(B,\tau)\rightarrow (B,\tau)$ such that $m(B\otimes S)\Delta =\eta\epsilon = m(S\otimes B)\Delta $ and the mentioned bimonad $B\tn -$ is left Hopf in this case and Hopf if $S$ is also invertible. A \emph{braided Hopf algebra} in a braided monoidal category $(\ct{C},\Psi) $ is just a central Hopf algebra $(H,\tau)$ where $\tau_{-}=\Psi_{H,-}$. We recover the usual notion of Hopf algebras as braided Hopf algebras in the braided monoidal category of vectorspaces. 

A Hopf monad $(T,\mu ,\eta , T_{2}, T_{0})$ on a monoidal category $\ct{C}$  is said to be \emph{augmented} if there exists a bimonad morphism $\xi : T\rightarrow \mathrm{id}_{\ct{C}}$, where the identity functor $\mathrm{id}_{\ct{C}}$ has trivial Hopf monad structure. For $\xi$ to be a bimonad morphism, $\xi \eta =\id  $, $\xi\mu = \xi T(\xi)$, $(\xi \tn \xi )T_{2}=\xi _{-\tn - }$ and $ \xi_{\un}= T_{0}$ must hold.
\begin{thm}\label{ThmAug} [Theorem 5.7 \cite{bruguieres2011hopf}] There is an equivalence of categories between the category of Hopf algebras in the center of $\mathcal{C}$ and augmented Hopf monads on $\mathcal{C}$.
\end{thm} 
An augmentation on a Hopf monad, provides $T(\un )$ with a central Hopf algebra structure and the Hopf monad $T$ in this case is shown to be isomorphic to the induced Hopf monad of $T(\un )$.  In particular, $(T(\un),T_{2}(\un ,\un), T_{0})$ forms a comonoid in $\ct{C}$, while $T(\un, \mu_{\un} \mathfrak{u}_{T(\un)},\eta_{\un})$ forms a monoid in $\ct{C}$, where $\mathfrak{u}_{X}= T(\xi_{X})(H^{l}_{\un,X})^{-1} (T(\un)\tn \eta_{X})$, and $\tau =(\xi\tn T(\un) )T_{2}(T(X),\un) \mathfrak{u}_{X}$ defines a braiding, which makes $(T(\un),\tau )$ a central Hopf algebra with its invertible antipode defined by $S =  \xi_{T(\un)}(H^{l}_{\un,\un})^{-1} (T(\un)\tn \eta_{\un})$ and $S^{-1} =  \xi_{T(\un)}(H^{r}_{\un,\un})^{-1} (\eta_{\un}\tn T(\un))$. We refer the reader to \cite{majid158algebras} for more details on braided Hopf algebras and \cite{bruguieres2011hopf} for the more details on augmented Hopf monads and the proof of Theorem \ref{ThmAug}. 

\section{Pivotal Objects and Pivotal Cover}\label{SPivOb}
In this section, we define the notion of pivotal objects, pairs and pivotal cover for arbitrary monoidal categories. Let $(\ct{C},\tn , \un )$ be a monoidal category. 

\begin{lemma}\label{LPiv} If $P$ is an object of $\ct{C}$, then the following statements are equivalent:
\begin{enumerate}[(I)]
\item The object $P$ is dualizable and there exists an isomorphism $\pr{P}\cong P^{\vee}$.
\item There exists an object $Q$ and morphisms $\cvl :\un\rightarrow P\tn Q$, $\evl : Q\tn P \rightarrow \un$ and $\cvr :\un\rightarrow Q\tn P$, $\evr : P\tn Q \rightarrow \un$, making $Q$ a left and right dual of $P$, respectively.
\item Left duals $\pr{P}$ and $\prescript{\vee\vee}{}{P}$ exist and there exists an isomorphism $P\cong \prescript{\vee\vee}{}{P}$. 
\end{enumerate}
We say $P$ is a \emph{pivotal object} if it satisfies any of the above statements and refer to an ordered pair $(P,Q)$, as in part (II), as a \emph{pivotal pair}.
\end{lemma}
\begin{proof} (I)$\Rightarrow$(II) Assume $P$ is dualizable with $\pr{P}$ and $P^{\vee}$, its left and right dual objects and $\cvl,\evl$ and $\cvr,\evr$ as the respective coevaluation and evaluation morphisms and let $f:\pr{P}\rightarrow P^{\vee}$ be an isomorphism. Hence, $(f^{-1} \tn P)\cvr$ and $\evr(P\tn f)$ make $\pr{P}$ right dual to $P$ and $Q= \pr{P}$ and $\cvl, \evl,  (f^{-1} \tn P)\cvr, \evr (P\tn f)$ satisfy the conditions in (II). 

(II)$\Rightarrow$(III) By assumption $Q=\pr{P}$ and $P=\pr{Q}$, thereby $P=\prescript{\vee\vee}{}{P}$.

(III)$\Rightarrow$(I) Let $f:P\rightarrow \prescript{\vee\vee}{}{P}$ be an isomorphism and $\cvl_{\pr{P}}: \un\rightarrow \prescript{\vee\vee}{}{P}\tn \pr{P}$, $\evl_{\pr{P}}:\prescript{\vee\vee}{}{P}\tn \pr{P}\rightarrow \un$ be the relevant coevaluation and evaluation morphisms. Hence, $( \pr{P}\tn f)\cvl_{\pr{P}}$ and $(f^{-1}\tn \pr{P})\evl_{\pr{P}}$ make $\pr{P}$ right dual to $P$, and $\pr{P}=P^{\vee}$.\end{proof}

Note that the coevaluation and evaluation morphisms making $Q$ a left and right dual of $P$ are part of the data for a pivotal pair $(P,Q)$. Additionally, $Q$ is also pivotal object by definition and $(Q,P)$ a pivotal pair with the duality morphisms swapped. Moreover, note that strong monoidal functors preserve pivotal objects, since they preserve duals and isomorphisms. In particular, a strong monoidal functor $F:\ct{C}\rightarrow \ct{D}$ sends a pivotal pair $(P,Q)$ in $\ct{C}$ to a pivotal pair $(F(P),F(Q))$ with the natural duality morphisms presented in Section \ref{SRigClos}.

We now review some examples of pivotal objects in monoidal categories. 
\begin{ex}\label{EBrai} Any dualizable object $P$ in a braided monoidal category $(\ct{B}, \Psi )$ is pivotal. Let $\pr{P}$ and $P^{\vee}$ be the left and right duals of $P$ with the coevaluation and evaluation maps, $\cvl, \evl $ and $\cvr , \evr$, respectively. In this case 
\begin{align}
\mathfrak{lr}_{P}:= (P^{\vee}\tn \evl \Psi_{P, \pr{P}})(\cvr\tn\pr{P}): \pr{P}\rightarrow P^{\vee}
\\  \mathfrak{rl}_{P}:=(\pr{P}\tn \evr)(\Psi^{-1}_{P,\pr{P}}\cvl\tn P^{\vee}):P^{\vee}\rightarrow \pr{P}
\end{align} 
are inverses and provide an isomorphism between $\pr{P}$ and $P^{\vee}$.
\end{ex}

\begin{ex}\label{Ebim} For a $\field$-algebra $A$, the category of $A$-bimodules, denoted by $\bim$ has a monoidal structure by tensoring bimodules over $A$. We say an $A$-bimodule is \emph{pivotal} if it is a pivotal object in $\bim$. In Section 4.2 of \cite{ghobadi2020hopf}, we provided a range of examples of first order differential calculi $(A, \Omega^{1}, d:A\rightarrow \Omega^{1})$, where $\Omega^{1}$ was a pivotal bimodule over $A$, the simplest examples being finitely generated free bimodules. Our examples also included Hopf bimodules over Hopf algebras. If $A$ is a Hopf algebra with an invertible antipode, the category of Hopf bimodules has braided monoidal structure and thereby any dualizable object in this category is pivotal. The forgetful functor from the category of Hopf bimodules to the category of bimodules over the Hopf algebra is strong monoidal. Hence as a bimodule, any dualizable Hopf bimodule is pivotal. 
\end{ex}

\begin{ex}\label{EEndC} For any category $\ct{C}$, the category of endofunctors on $\ct{C}$, denoted by $\mathrm{End}(\ct{C})$, has a monoidal structure via composition of functors i.e. $F\tn G = FG$ for $F,G\in \mathrm{End}(\ct{C})$ and the identity functor $\id $ acting as the monoidal unit. In this case, a functor $F$ being left (right) dual to $G$, is exactly equivalent to $F$ being left (right) adjoint to $G$. An endofunctor is thereby called pivotal if it has left and right adjoint functors which are isomorphic. Such adjunctions are referred to as \emph{ambidextrous} and the monad $GF$ on $\ct{C}$ is called a \emph{Frobenius} monad [Theorem 17 \cite{lauda2006frobenius}]. 
\end{ex}
A rigid monoidal category $\ct{C}$ is called \emph{pivotal} (sometimes called \emph{sovereign}) if there exists a natural isomorphism $ \varrho : \mathrm{id}_{ \ct{C}}\rightarrow \prescript{\vee \vee}{}{(-)} $. Equivalently, the condition is sometimes stated as the existence of a monoidal natural isomorphism $ \varrho^{\vee} : (-)^{\vee }\rightarrow \pr{(-)} $. In \cite{shimizu2015pivotal}, the pivotal cover of a rigid monoidal category was introduced by K. Shimizu. Independently, we discovered this notion for general monoidal categories, by encountering the notion of \emph{pivotal morphisms} between pivotal objects in \cite{ghobadi2020hopf}.
\begin{defi}\label{DPivMo} Let $\ct{C}$ be a monoidal category, and $(P_{1}, Q_{1})$ and $(P_{2}, Q_{2})$ pivotal pairs in $\ct{C}$ with $\cvl_{i}, \evl_{i}, \cvr_{i}, \evr_{i}$ being the relevant coevaluation and evaluation morphisms. We say a morphism $f: P_{1}\rightarrow P_{2}$ is \emph{pivotal} if 
\begin{equation}\label{EqPivMo}
(\evl_{2}\tn Q_{1} ) ( Q_{2} \tn f \tn Q_{1} ) (Q_{2}\tn \cvl_{1}) = (Q_{1}\tn \evr_{2} ) (Q_{1}\tn f \tn Q_{2} ) (\cvr_{1}\tn Q_{2} )
\end{equation}
as morphisms from $Q_{2}$ to $Q_{1}$. The \emph{pivotal cover} of $\ct{C}$, denoted by $\ct{C}^{piv}$, has ordered pivotal pairs $(P,Q)$, in $\ct{C}$, as objects and pivotal morphisms of $\ct{C}$ as morphisms. 
\end{defi}
There is a subtlety which we must point out. Namely, that a morphisms $f:P_{1}\rightarrow P_{2}$ being pivotal does really depend on the choice of $Q_{1}$ and $Q_{2}$. Even for a single pivotal object $P$ and two $Q_{1}$ and $Q_{2}$ left and right duals of $P$ with $\cvl_{i}, \evl_{i}, \cvr'_{i}, \evr'_{i}$ denoting the relevant coevaluation and evaluation morphisms, the identity morphism $\mathrm{id}_{P}$ is not necessarily a morphism between pivotal pairs $(P,Q_{1})$ and $(P,Q_{2})$. There exist isomorphisms $h^{l}: Q_{1}\rightarrow Q_{2}$ and $h^{r}: Q_{1}\rightarrow  Q_{2}$ defined by $h^{l}=(\evl_{2}\tn Q_{1} )( Q_{2}\tn \cvl_{1})$ and $h^{r}=( Q_{1} \tn \evr_{2} ) (\cvr_{1}\tn Q_{2} )$. The identity morphism being pivotal in terms $Q_{1}$ and $Q_{2}$ is exactly equivalent to $h^{l}=h^{r}$, which is not necessarily true (Example \ref{EHopfMany}). However, the identity morphism is clearly a pivotal morphism from $(P, Q_{1} )$ to itself. It should also be clear that $\ct{C}^{piv}$ is well-defined and pivotal morphisms are closed under composition. 

In the terminology of rigid monoidal categories, if $Q= \pr{P}= P^{\vee}$ is the chosen left and right dual of $P$, the left hand morphism in \ref{EqPivMo} is exactly $\pr{f}$ and the right hand morphism, $ f^{\vee}$, and we call a morphism $f$ between objects pivotal, if $\pr{f}=f^{\vee}$. 

In Section 5.2 of \cite{ghobadi2020hopf}, we presented two families of examples of pivotal morphisms. In \cite{ghobadi2020hopf}, we were interested in differential graded algebras $(\oplus_{i\geq 0} \Omega^{i},d,\wedge ) $, with $\wedge$ denoting the graded multiplication, where $\Omega^{1}$ and $\Omega^{2}$ were pivotal bimodules over the algebra $(\Omega_{0}, \wedge)$, and $\wedge : \Omega^{1}\tn \Omega^{1}\rightarrow \Omega^{2}$ was a pivotal morphism. One of the examples presented in \cite{ghobadi2020hopf} was that of Woronowicz's bicovariant algebras \cite{woronowicz1989differential} over the group algebra, for any arbitrary group. The Woronowicz construction extends any Hopf module calculus over a Hopf algebra to a DGA of Hopf bimodules and has been generalised to the language of braided abelian categories by Majid, Section 2.6 \cite{beggs2019quantum}, in the name of \emph{braided exterior algebras}. In the following examples we show the mentioned morphism is always pivotal for any braided exterior algebra.

\begin{ex}\label{EPivMo}[Braided Exterior Algebra] Let $(\ct{C},\tn,\un )$ be a braided monoidal category and $\Psi$ denote its braiding. Recall from Example \ref{EBrai} that any dualizable object in $\ct{C}$ is pivotal. Let $P_{1}$ and $P_{2}$ be dualizable objects in $\ct{C}$ such that $\pi :P_{1}\tn P_{1}\rightarrow P_{2}$ is the coequalizer of the parallel pair $ \mathrm{id}_{P_{1}\tn P_{1}},\Psi_{P_{1},P_{1}} : P_{1}\tn P_{1} \rightrightarrows P_{1}\tn P_{1} $. We claim that $\pi$ is a pivotal morphism. As in Example \ref{EBrai}, $\pr{P_{1}}$ and $\pr{P}_{2}$ denote left duals of $P_{1}$ and $P_{2}$, respectively, which become right duals by the isomorphism provided, so that $\pr{P_{1}}\tn \pr{P_{1}}$ also becomes a left and right dual of $P_{1}\tn P_{1}$. Hence, writing the pivotal condition, \ref{EqPivMo}, for $\wedge$ in terms of $\pr{P_{1}}\tn \pr{P_{1}}$ and $\pr{P_{2}}$, reduces to checking if the morphisms 
\begin{align*}
&(\evl_{2}(\pr{P_{1}})(\pr{P_{1}} )) ( \pr{P_{2}}  \wedge (\pr{P_{1}})( \pr{P_{1}} )) \big(\pr{P_{2}}(P_{1}  \cvl_{1} \pr{P_{1}})(\cvl_{1} ) \big), 
\\ &(\pr{P_{1}}\pr{P_{1}}\evr_{2} ) ((\mathfrak{rl}_{P_{1}}\tn \mathfrak{rl}_{P_{1}}\tn \wedge )  {P_{2}}^{\vee}) \big((P_{1}^{\vee}  \cvr P_{1})(\cvr )\tn \mathfrak{lr}_{P_{2}} \big)
\end{align*}
are equal. By the definition of $\mathfrak{rl}$ and $\mathfrak{lr}$ and the properties of the braiding, the second morphism simplifies as follows
\begin{align*}
(\pr{P_{1}} &\pr{P_{1}} \evl_{2} \Psi_{P_{2}, \pr{P_{2}}})((\pr{P_{1}}) (\pr{P_{1}}) \wedge  (\pr{P_{2}})) (\pr{P_{1}}  \Psi^{-1}_{P_{1},\pr{P_{1}}}\cvl_{1}  P_{1}(\pr{P_{2}}))
\\ &(\Psi^{-1}_{P_{1},\pr{P_{1}}}\cvl_{1} \pr{P_{2}} )
\\ =&(\evl_{2} (\pr{P_{1}})( \pr{P_{1}}))(\pr{P_{2}} \Psi_{\pr{P_{1}} \tn \pr{P_{1}}, P_{2}} )  (( \pr{P_{2}})( \pr{P_{1}}) (\pr{P_{1}} )\wedge ) 
\\ &\big(\pr{P_{2}}(\pr{P_{1}}  \Psi^{-1}_{P_{1},\pr{P_{1}}}\cvl_{1}  P_{1})(\Psi^{-1}_{P_{1},\pr{P_{1}}}\cvl_{1} ) \big)
\\ = &(\evl_{2} (\pr{P_{1}})( \pr{P_{1}}))(\pr{P_{2}} \Psi_{\pr{P_{1}} , P_{2}} (\pr{P_{1}})) (\pr{P_{2}} (\pr{P_{1}})\wedge\Psi^{-1}_{P_{1},P_{1}} (\pr{P_{1}}))
\\& \big(\pr{P_{2}}( \Psi^{-1}_{P_{1},\pr{P_{1}}}\cvl_{1} \tn \cvl_{1} )\big)
\\ = & (\evl_{2}(\pr{P_{1}})(\pr{P_{1}} ))(\pr{P_{2}} \wedge\Psi_{P_{1},P_{1}}^{-1} \Psi_{P_{1},P_{1}}^{-1} (\pr{P_{1}})(\pr{P_{1}} ))( \pr{P_{2}} (P_{1}\cvl_{1} \pr{P_{1}} )\cvl_{1}) 
\\=&(\evl_{2}(\pr{P_{1}})(\pr{P_{1}} )) ( \pr{P_{2}}  \wedge (\pr{P_{1}})( \pr{P_{1}} )) \big(\pr{P_{2}}(P_{1}  \cvl_{1} \pr{P_{1}})(\cvl_{1} ) \big)
\end{align*}
The calculations above are much more clear from a pictorial point of view and readers who are familiar with the graphical calculus of braided monoidal categories, should draw the said morphisms for a quicker proof. We refer the reader to Chapter 2 of \cite{turaev2017monoidal} and Section 2.6 of \cite{beggs2019quantum}, for more details on graphical calculus. 
\end{ex}

We can define a natural monoidal structure on $\ct{C}^{piv}$ by $(P_{1},Q_{1})\tn (P_{2},Q_{2}) = (P_{1}\tn P_{2}, Q_{2}\tn Q_{1})$, since the tensor of pivotal morphisms is again a pivotal morphism in this way. Thereby, $\ct{C}^{piv}$ lifts the monoidal structure of $\ct{C}$, so that the natural forgetful functor $H:\ct{C}^{piv}\rightarrow \ct{C} $, which sends a pivotal pair $(P,Q)$ to $P$, is strict monoidal. Furthermore, notice that $\ct{C}^{piv}$ is also rigid and admits left and right duality functors $\pr{(-)}= (-)^{\vee} : \ct{C}^{piv} \rightarrow \ct{C}^{piv} $, which are defined by $(P,Q)^{\vee}=(Q,P)$ and $f^{\vee} = (\evl_{2}\tn Q_{1} ) ( Q_{2} \tn f \tn Q_{1} ) (Q_{2}\tn \cvl_{1}) $ for $f: (P_{1},Q_{1})\rightarrow (P_{2},Q_{2})$, as in Definition \ref{DPivMo}. By our definition of pivotal morphisms, it should be clear that the trivial identity morphism forms an isomorphism between $\pr{(-)}$ and $(-)^{\vee}$ and the category $\ct{C}^{piv}$ is trivially pivotal. 

In order to discuss the universal property of the pivotal cover, we recall the definition of a \emph{pivotal functor} from \cite{shimizu2015pivotal}. If $\ct{C}$ and $\ct{D}$ are rigid monoidal categories and $F: \ct{C}\rightarrow \ct{D}$ is a strong monoidal functor, then we have a natural family of unique isomorphisms $\zeta  : F(\pr{-}) \rightarrow \pr{F(-)}$ defined by 
\begin{align*}
\zeta_{X}&=(F_{0}F(\evl)F_{2}(\pr{X},X) \tn \pr{F(X)}) (F(\pr{X})\tn\cvl_{F(X)})
\\ \zeta_{X}^{-1} &= (\evl_{F(X)}\tn F(\pr{X}) )(\pr{F(X)}\tn F^{-1}_{2}(X,\pr{X})  F(\cvl_{X}) F_{0}^{-1}) 
\end{align*}
where $X$ is an object of $\ct{C}$. If $\ct{C}$ and $\ct{D}$ are pivotal categories with pivotal structures $\varrho^{\ct{C}}  : \mathrm{id}_{ \ct{C}}\rightarrow \prescript{\vee \vee}{}{(-)} $ and $\varrho^{\ct{D}} : \mathrm{id}_{ \ct{C}}\rightarrow \prescript{\vee \vee}{}{(-)} $, we say $F$ \emph{preserves the pivotal structure} if 
\begin{equation}\label{EqPivFun}
\varrho^{\ct{D}}_{F(X)} = \pr{(\zeta_{X}^{-1})} (\zeta_{\pr{X}} ) F( \varrho^{\ct{C}}_{X})
\end{equation}
holds for all objects $X$ of $\ct{C}$. 
\begin{thm}\label{TCPiv} The pivotal cover $\ct{C}^{piv}$ of $\ct{C}$ is pivotal and satisfies the following universal property: if $\ct{D}$ is a pivotal monoidal category and $G:\ct{D}\rightarrow \ct{C}$ a strong monoidal functor, then there exists a unique functor $G' : \ct{D}\rightarrow \ct{C}^{piv} $ so that $G=HG' $ and $G'$ respects the pivotal structures. 
\end{thm}
\begin{proof} Let the pivotal structure of $\ct{D}$ be denoted by the natural isomorphism $\varrho: \mathrm{id}_{ \ct{C}}\rightarrow \prescript{\vee \vee}{}{(-)} $. Since strong monoidal functors preserve duals, $G(X)$ for any object $X$ in $\ct{D}$ will be a pivotal object in $\ct{C}$. In particular, we define $G'(X)$ to be the pivotal pair $(G(X),G(\pr{X}))$ with coevaluation and evaluation morphisms 
\begin{align*}
\cvl^{G}_{X}&:= G_{2}^{-1}(X,\pr{X}) G(\cvl_{X}) G_{0}^{-1}, \hspace{0.4cm}  G_{2}^{-1}(\pr{X},X) G((\pr{X} \tn \varrho_{X
} ) \cvl_{\pr{X}} ) G_{0}^{-1} 
 \\ \evl^{G}_{X}&:=G_{0}G(\evl_{X})G_{2}(\pr{X},X),\hspace{1.2cm}  G_{0}G(\evl_{\pr{X}} (\varrho^{-1}_{X}\tn \pr{X} ) ) G_{2}( X,\pr{X}) 
\end{align*}
If $f:X_{1}\rightarrow X_{2}$ is a morphism in $\ct{D}$, then $G(f)$ is a pivotal morphism between $(G(X_{1}),G(\pr{X_{1}}))$ and $(G(X_{2}),G(\pr{X_{2}})) $, since $\varrho_{\pr{X}}^{-1} f^{\vee} \varrho_{\pr{X}} =\pr{f} $. Hence, by letting $G'(f)= G(f)$ we have defined a functor $G':\ct{D}\rightarrow \ct{C}^{piv}$ such that $HG'=G$. 

Additionally, $G'$ is a pivotal functor: if $\zeta  : G'(\pr{-}) \rightarrow \pr{G(-)}$ is the unique natural isomorphism as defined before the Theorem, then $G'(\pr{X})= (G(\pr{X}), G(\prescript{\vee \vee}{}{X}) ) $ and $\pr{G(X)}= (G(\pr{X}),G(X))$ with the appropriate duality morphisms as defined above. Hence,
\begin{align*}
& \pr{(\zeta_{X}^{-1})}(\zeta_{\pr{X}} ) G'( \varrho_{X}) =\pr{(\zeta_{X}^{-1})} (\evl^{G}_{\pr{X}} \tn \pr{G'(X)}) (G'( \varrho_{X})\tn\cvl_{G'(\pr{X})})
\\&= (\evl_{G'(\pr{X})} \tn \prescript{\vee \vee}{}{G'(X)} )(\pr{G'(X)}\tn \zeta_{X}^{-1}\tn \prescript{\vee \vee}{}{G'(X)})(\pr{G'(X)}\tn \cvl_{\pr{G'(X)}} )
\\ &\quad (\evl^{G}_{\pr{X}} \tn \pr{G'(X)}) (G'( \varrho_{X})\tn\cvl_{G'(\pr{X})})
\\&= (\evl_{G'(\pr{X})} \tn \prescript{\vee \vee}{}{G'(X)} )(\pr{G'(X)}\tn \evl_{G'(X)}\tn G'(\pr{X}) \tn \prescript{\vee \vee}{}{G'(X)} )
\\ &\quad (\pr{G'(X)}\tn\pr{G'(X)}\tn \cvl^{G}_{X} \tn \prescript{\vee \vee}{}{G'(X)})(\pr{G'(X)}\tn \cvr_{G'(X)} ) 
\\&\quad (\evl^{G}_{\pr{X}} \tn \pr{G'(X)}) (G'( \varrho_{X})\tn\cvl_{G'(\pr{X})})
\end{align*}
and by construction $\mathrm{id}_{(P,Q)}= \mathrm{P}$ for any pair $(P,Q)$ in $\ct{C}^{piv}$ and
\begin{align*}
&\pr{(\zeta_{X}^{-1})}(\zeta_{\pr{X}} ) G'( \varrho_{X}) = (\evl^{G}_{\pr{X}} \tn G(X) )(G(\pr{X})\tn \cvr_{G'(X)} ) (\evl^{G}_{\pr{X}} \tn G(\pr{X}))
\\&\quad  (G( \varrho_{X})\tn\cvl^{G}_{\pr{X}})
\\ &=(\evl^{G}_{\pr{X}} \tn G(X) ) (G(\pr{X})\tn G(X) \tn G( \varrho^{-1}_{X})) (G'(\pr{X})\tn \cvl^{G}_{\pr{X}} ) G( \varrho_{X})
\\ &=\mathrm{id}_{G(X)}= \mathrm{id}_{G'(X)}
\end{align*}
holds and thereby $G'$ is pivotal. 
\end{proof} 
In \cite{shimizu2015pivotal}, the pivotal cover of a left rigid monoidal category $\ct{C}^{piv}$ is constructed as the category of \emph{fixed objects} by the endofunctor $\prescript{\vee \vee}{}{(-)}:\ct{C}\rightarrow \ct{C}$. Constructing the pivotal cover as such has two main drawback, namely that we need to assume all objects in $\ct{C}$ have left duals and $\ct{C}$ is left rigid so that there is a distinguished choice of left dual for ever object. While we will not directly compare the constructions, the universal property above, is also proved in Theorem 4.3 of \cite{shimizu2015pivotal} and thereby the two constructions of $\ct{C}^{piv}$ are equivalent when $\ct{C}$ is left rigid.

\section{The Category $\ct{C}(P,Q)$}\label{SC(P)}
Given a pivotal pair $P$ and $Q$, as in Lemma \ref{LPiv} (II), we define the category of $P$ and $Q$ \emph{intertwined objects}, denoted by $\ct{C}(P,Q)$, as the category whose objects are pairs $(X,\sigma)$, where $X$ is an object of $\ct{C}$ and $\sigma: X\tn P \rightarrow P\tn X$ an invertible morphism in $\ct{C}$ such that 
\begin{align}
(\mathrm{ev}\otimes X\otimes Q)(Q\otimes\sigma\otimes Q)(Q\otimes X\otimes \mathrm{coev}): Q\otimes X\rightarrow X\otimes Q \label{Eqovsig1} &
\\ (Q\otimes X\otimes \evr )(Q\otimes\sigma^{-1}\otimes Q)(\cvr\otimes X\otimes Q): X\otimes Q \rightarrow Q\otimes X \label{Eqovsig2}&
\end{align}
are inverses. Morphisms between objects $(X,\sigma )$, $(Y,\tau)$ of $\ct{C}(P,Q)$ are morphisms $f:X\rightarrow Y$ in $\ct{C}$, which satisfy $\tau (f\tn P)= (P\tn f) \sigma$. For an object $(X,\sigma)$ in $\ct{C}(P,Q)$, we call $\sigma$ a $P$-\emph{intertwining} and denote the induced morphisms \ref{Eqovsig1} and \ref{Eqovsig2}, by $\ov{\sigma}$ and $\ov{\sigma}^{-1}$, respectively, and call them \emph{induced} $Q$\emph{-intertwinings}.

Observe that the definition of $\ct{C}(P,Q)$ is dependent on the choice of $Q$: let $P$ and $Q'$ together with $\cvl' :\un\rightarrow P\tn Q'$, $\evl' : Q'\tn P \rightarrow \un$ and $\cvr' :\un\rightarrow Q'\tn P$, $\evr' : P\tn Q' \rightarrow \un$ satisfy the conditions of Lemma \ref{LPiv} (II). Hence, we have two induced isomorphisms between $Q$ and $Q'$, $f = (\evl\tn Q')(Q \tn \cvl' )$ and $ f^{-1}= (\evl'\tn Q)(Q'\tn \cvl )$, and $g= (Q'\tn \evr)(\cvr'\tn Q)$ and $g^{-1}= (Q\tn \evr' )(\cvr\tn Q')$. Additionally, if for a $P$-intertwining $(X,\sigma)$, we denote the induced $Q$-intertwinings and induced $Q' $-intertwinings by $\ov{\sigma}_{Q}$, $\ov{\sigma}_{Q}^{-1}$ and $\ov{\sigma}_{Q'}$, $\ov{\sigma}_{Q'}^{-1}$, respectively, then $\ov{\sigma}_{Q'} = (g\tn X)\ov{\sigma}_{Q}(X\tn g^{-1})$ and $ \ov{\sigma}_{Q'}^{-1} = ( X\tn f) \ov{\sigma}_{Q}^{-1}(f^{-1}\tn X)$. Hence, $\ov{\sigma}_{Q}$ and $\ov{\sigma}_{Q}^{-1}$ being inverses is not equivalent to $\ov{\sigma}_{Q'}$ and $\ov{\sigma}_{Q'}^{-1}$ being inverses unless $f=g$. 

On the other hand, the category $\ct{C}(Q,P)$ is isomorphic to $\ct{C}(P,Q)$. The isomorphism sends an object $(X,\sigma )$ in $\ct{C}(P,Q)$ to $(X,\ov{\sigma}^{-1}) $ in $\ct{C}(Q,P)$. The $Q$-intertwining $\ov{\sigma}^{-1}$ is invertible and the induced $P$-intertwinings on $X$ in $\ct{C}(Q,P)$ are precisely $\sigma$ and $\sigma^{-1}$: 
\begin{align*}
\sigma^{-1}=&(\evr\otimes X\otimes P)(P\otimes\ov{\sigma}^{-1}\otimes P)(P\otimes X\otimes \cvr)
\\ \sigma  =&(P\otimes X\otimes \evl )(P\otimes\ov{\sigma}\otimes P)(\cvl\otimes X\otimes P)
\end{align*}
Note that the isomorphism described between $\ct{C}(Q,P)$ and $\ct{C}(P,Q)$ commutes with the forgetful functors from each category to $\ct{C}$.  

The monoidal structure of $\ct{C}$ lifts to $\ct{C}(P,Q)$ so that the forgetful functor $U:\ct{C}(P,Q)\rightarrow \ct{C}$ which sends a pair $(X,\sigma)$ to its underlying object $X$, becomes strict monoidal: for any pair of objects $(X,\sigma)$ and $(Y,\tau )$, the monoidal structure of $\ct{C}(P,Q)$, denoted by $\tn$ again, is defined by 
\begin{equation}\label{EqMonoidal}
(X,\sigma ) \tn (Y ,\tau ) =  \big(X\tn Y, (\sigma\tn Y)(X\tn \tau)  \big)
\end{equation}
and $(\un , \mathrm{id}_P)$ acts as the monoidal unit. Furthermore, $\tn$ is defined on pairs of morphisms of $\ct{C}(P,Q)$, as it is by $\tn$ in $\ct{C}$.

Our construction is very similar to that of the center of a monoidal category, and is an example of its generalisation, the dual of a strong monoidal functor \cite{majid1992braided}, which we will comment on in Section \ref{SDual}. 

\begin{thm}\label{TMon} The monoidal structure on $\ct{C}(P,Q)$, as described above, is well-defined.  
\end{thm}
\begin{proof} The only non-trivial fact we need to check in our case is whether $(X,\tau ) \tn (Y ,\tau )$ is an object of $\ct{C}(P,Q)$. In particular, if $\sigma$ and $\tau$ are invertible, it should be clear that $ (\sigma\tn Y)(X\tn \tau) $ is also invertible, however, we need to prove that the induced $Q$-intertwinings, \ref{Eqovsig1} and \ref{Eqovsig2}, for $X\tn Y$ are inverses. This follows from the fact that the induced $Q$-intertwinings, \ref{Eqovsig1} and \ref{Eqovsig2} for $(X,\tau )$ and $ (Y ,\tau )$ are inverses:  
\begin{align*}
&\hspace{-0.5cm}(\ov{\sigma\tn\tau} )(\ov{\sigma\tn\tau}^{-1} )=(\mathrm{ev} XYQ)(Q\sigma\tn \tau Q)(QXY \mathrm{coev})  (QXY\evr )(Q(\sigma\tn \tau)^{-1}Q)
\\ & (\cvr XYQ)
\\ = &(\mathrm{ev} XYQ)(Q\sigma YQ) (QX \tau Q)(QXY \mathrm{coev})  (QXY\evr )(QX\tau^{-1}Q)(Q\sigma^{-1}YQ)
\\ & (\cvr XYQ)
\\ = &(\mathrm{ev} XYQ)(Q\sigma YQ) (QXP \evl YQ) (QX\cvl PYQ)(QX \tau Q)(QXY \mathrm{coev})  
\\ & (QXY\evr )(QX\tau^{-1}Q)(QX\evr YQ)(QXP\cvr YQ)(Q\sigma^{-1}YQ)(\cvr XYQ)
\\= & (X\evl YQ)(XQ\tau Q)(XQY\cvl)(\evl XQY)(Q\sigma QY)(QX\cvl Y)
\\ & (QX\evr Y)(Q\sigma^{-1}QY)(\cvr XQY)(XQY\evr )(XQ\tau^{-1}Q)(X\cvr YQ)
\\=& (X\evl YQ)(XQ\tau Q)(XQY\cvl)(XQY\evr )(XQ\tau^{-1}Q)(X\cvr YQ)= \mathrm{id}_{XYQ}
\end{align*}
In a symmetric manner, it follows that $(\ov{\sigma\tn\tau}^{-1} )(\ov{\sigma\tn\tau} )=\mathrm{id}_{QXY}$.\end{proof}
\begin{thm}\label{TCld} If $\ct{C}$ is a left closed monoidal category, then $\ct{C}(P,Q)$ has a left closed monoidal structure which lifts that of $\ct{C}$ and the forgetful functor $U$ is left closed.
\end{thm} 
\begin{proof} Let $(A,\sigma_{A})$ and $(B,\sigma_{B})$ be objects in $\ct{C}(P,Q)$. If $\ct{C}$ is left closed, we denote the right adjoint functor to $-\tn A$, by $[A,-]^{l}$, and let $\eta^{A}_{-}: -\rightarrow [A,-\tn A]^{l} $ and $\epsilon^{A}_{-}: [A,-]^{l}\tn A\rightarrow -$ denote the unit and counit of this adjunction. To demonstrate that the left closed structure of $\ct{C}$ lifts to $\ct{C}(P,Q)$, we provide a functorial $P$-intertwining on $[A,B]^{l}$, and demonstrate that the unit and counit morphisms are morphisms in $\ct{C}(P,Q)$. We claim that $\langle \sigma_{A},\sigma_{B} \rangle_{l} $ as defined below is a $P$-intertwining with the described inverse:
\begin{align*}
 \langle \sigma_{A},\sigma_{B} \rangle_{l}  :=& (P[A,(\evl B)(Q\sigma_{B})(Q\epsilon^{A}_{B}P)(Q[A,B]^{l}\sigma^{-1}_{A})]^{l})(P\eta^{A}_{Q[A,B]^{l}P})
 \\&(\cvl[A,B]^{l}P)
\\\langle \sigma_{A},\sigma_{B} \rangle_{l}^{-1}  := &([A,(B\evr )(\sigma^{-1}_{B}Q)(P\epsilon^{A}_{B}Q)(P[A,B]^{l}\ov{\sigma_{A}})]^{l}P)(\eta^{A}_{P[A,B]^{l}Q}P)
\\&(P[A,B]^{l}\cvr)
\end{align*}
Demonstrating that $([A,B]^{l},\langle \sigma_{A},\sigma_{B} \rangle_{l})$ is an object of $\ct{C}(P,Q)$ requires showing that $\langle \sigma_{A},\sigma_{B} \rangle_{l} $ and $\langle \sigma_{A},\sigma_{B} \rangle_{l}^{-1} $ are inverses and that the induced $Q$-intertwinings given by 
\begin{align*}
 \ov{\langle \sigma_{A},\sigma_{B} \rangle_{l}}  =& ([A,(\evl B)(Q\sigma_{B})(Q\epsilon^{A}_{B}P)(Q[A,B]^{l}\sigma^{-1}_{A})]^{l}Q)(\eta^{A}_{Q[A,B]^{l}P}Q)
 \\&(Q[A,B]^{l}\cvl)
\\\ov{\langle \sigma_{A},\sigma_{B} \rangle_{l}}^{-1}  = &(Q[A,(B\evr )(\sigma^{-1}_{B}Q)(P\epsilon^{A}_{B}Q)(P[A,B]^{l}\ov{\sigma_{A}})]^{l})(Q\eta^{A}_{P[A,B]^{l}Q})
\\&(\cvr [A,B]^{l}Q)
\end{align*}
are inverses as well. These fact are not hard to show, but reduce to long diagram chases, for which we refer the reader to Section \ref{SDiag}. If $f: (B,\sigma_{B})\rightarrow (C,\sigma_{C})$ is a morphism in $\ct{C}(P,Q)$, then it follows by definition that $[A,f]^{l}$ is also a morphism in $\ct{C}(P,Q)$:
\begin{align*}
&(P[A,f]^{l})\langle \sigma_{A},\sigma_{B} \rangle_{l} = (P[A,f]^{l})(P[A,(\evl B)(Q\sigma_{B})(Q\epsilon^{A}_{B}P)(Q[A,B]^{l}\sigma^{-1}_{A})]^{l})
\\&(P\eta^{A}_{Q[A,B]^{l}P})(\cvl[A,B]^{l}P)
\\=& (P[A,f(\evl B)(Q\sigma_{B})(Q\epsilon^{A}_{B}P)(Q[A,B]^{l}\sigma^{-1}_{A})]^{l})(P\eta^{A}_{Q[A,B]^{l}P})(\cvl[A,B]^{l}P)
\\=& (P[A,(\evl C)(Q\sigma_{C})(QfP)(Q\epsilon^{A}_{B}P)(Q[A,B]^{l}\sigma^{-1}_{A})]^{l})(P\eta^{A}_{Q[A,B]^{l}P})(\cvl[A,B]^{l}P)
\\=& (P[A,(\evl C)(Q\sigma_{C})(Q\epsilon^{A}_{C}P)(Q[A,C]^{l}\sigma^{-1}_{A})(Q[A,f]^{l}PA)]^{l}) (P\eta^{A}_{Q[A,B]^{l}P})
\\&(\cvl[A,B]^{l}P)= \langle \sigma_{A},\sigma_{C} \rangle([A,f]^{l}P)
\end{align*}
Hence the assignment
\begin{align*}
[(A,\sigma_{A}), - ]^{l} :\ct{C}(&P)\longrightarrow \ct{C}(P,Q)
\\  (B,&\sigma_{B}) \mapsto  \big([A,B]^{l}, \langle \sigma_{A},\sigma_{B}\rangle \big) 
\end{align*}
is functorial by acting as $[A,-]^{l}$ on morphisms. In particular, the functor $[(A,\sigma_{A}),-]^{l}$ is a lift of $[A,-]^{l}$ via the forgetful functor $U$ so that $U[(A,\sigma_{A}),-]^{l}= [A,U(-)]^{l}$. Hence, it only remains to show that natural transformations $\eta^{A} $ and $\epsilon^{A}$ lift to $\ct{C}(P,Q)$, with respect to the defined $P$-intertwinings on $[(A,\sigma_{A}),-]^{l}$. We must check that 
$$ (P\eta_{B}^{A})\sigma_{B} =  \langle \sigma_{A},\sigma_{B}\tn \sigma_{A}\rangle (\eta_{B}^{A} P):B\tn P \rightarrow P\tn  [A,B\tn A]^{l} $$
holds. We must also check that the counit commutes with the $P$-intertwinings i.e.
$$ \sigma_{B}(\epsilon_{B}^{A}P)= (P\epsilon_{B}^{A})(\langle \sigma_{A},\sigma_{B}\rangle\tn\sigma_{A}):[A,B]^{l}\tn A \tn P \rightarrow P \tn B$$
holds. Both fact are proved in the form of commutative diagram, which are presented in Section \ref{SDiag}. Hence, we have demonstrated that the left closed structure of $\ct{C}$ lifts to $\ct{C}(P,Q)$ via the forgetful functor $U$. 
\end{proof}

\begin{cor}\label{CRCld} If $\ct{C}$ is a right closed monoidal category, then $\ct{C}(P,Q)$ has a right closed monoidal structure which lifts that of $\ct{C}$ and the forgetful functor $U$ is right closed.
\end{cor}
\begin{proof} One could prove this statement directly as done for the left closed structure in Theorem \ref{TCld}, however, we take a short-cut in this case. Notice that a right closed structure on $(\ct{C},\tn )$ corresponds to a left closed structure on $(\ct{C},\tn^{\op})$. Hence, the forgetful functor $U:\ct{C}(P,Q)\rightarrow \ct{C}$ lifting the right closed structure of $\ct{C}$ is equivalent to $U^{\op}:\ct{C}(P,Q)^{\op}\rightarrow \ct{C}^{\op}$ lifting the left closed structure of $\ct{C}^{\op}$, where by $\ct{C}^{\op}$ we mean $(\ct{C},\tn^{\op})$. On the other hand, we observe that $(P,Q)$ is again a pivotal pair in $\ct{C}^{\op}$ with $\evr$ and $\cvr$ making $Q$ a left dual of $P$ and $\evl$ and $\cvl$ making $Q$ a right dual of $P$ in $(\ct{C},\tn^{\op})$. Furthermore, we have an isomorphism of categories 
\begin{equation*}
\xymatrix@R-25pt{L:\ct{C}^{\op}( P,Q)\longrightarrow \ct{C}(P,Q)^{\op} 
\\ \quad(A, \sigma : A\tn^{\op} P \rightarrow P\tn^{\op} A )\mapsto (A,\sigma^{-1}: A\tn P\rightarrow P\tn A)   } 
\end{equation*}
\begin{equation*}
\xymatrix@R-25pt{R:\ct{C}(P,Q)^{\op}\longrightarrow \ct{C}^{\op}(P,Q ) 
\\ \hspace{1cm}( A,\sigma :A\tn P\rightarrow P\tn A) \mapsto (A,\sigma^{-1}:A\tn^{\op} P\rightarrow P\tn^{\op} A)}
\end{equation*}
which is monoidal i.e. for a pair of objects $(A,\sigma_{A}) $ and $(B,\sigma_{B}) $ in $\ct{C}(P,Q)^{\op}$, we have 
\begin{align*}
 R \big( (A,\sigma_{A}) \tn^{\op}(B,\sigma_{B}) \big)&= R\big( (B\tn A ,\sigma_{B}\tn \sigma_{A} )\big) = \big(B\tn A,(\sigma_{B}\tn \sigma_{A} )^{-1}\big)
 \\ &= \big( B\tn A , (B\tn \sigma_{A}^{-1} )( \sigma_{B}^{-1}\tn A) \big) 
 \\ &= ( A,\sigma_{A}^{-1}) \tn^{\op} (B, \sigma^{-1}) = R( (A,\sigma_{A})) \tn^{\op} R((B,\sigma_{B}))
\end{align*}
Moreover, $U^{\op}L$ is precisely the forgetful functor from $\ct{C}^{\op}(P ,Q)$ to $\ct{C}^{\op}$ which sends a pair $(A,\sigma)$ to $A$. By Theorem \ref{TCld}, we know that $U^{\op}L$ lifts the left closed structure of $\ct{C}^{\op}$ and since $L$ is a strict isomorphism of monoidal categories, we conclude that $U^{op}$ also lifts the left closed structure of $\ct{C}^{\op}$. 
\end{proof}
The proof of Corollary \ref{CRCld} allows us to compute the induced $P$-intertwinings on the right inner homs of $\ct{C}$ so that the right closed structure of $\ct{C}$ lifts to $\ct{C}(P,Q)$. Explicitly, if $(A,\sigma_{A})$ and $(B,\sigma_{B})$ are objects in $\ct{C}(P,Q)$, and $\Gamma^{A}_{-}: -\rightarrow [A,A\tn -]^{r} $ and $\Theta^{A}_{-}: A\tn [A,-]^{r}\rightarrow -$ denote the unit and counit of $-\tn A \dashv [A,-]^{r}$ in $\ct{C}$, then 
\begin{align*}
\langle \sigma_{A},\sigma_{B} \rangle_{r}  :=& (P[A,(\evl B)(Q\sigma_{B})(Q\Theta^{A}_{B}P)(\ov{\sigma}_{A}^{-1}[A,B]^{r}P)]^{r})(P\Gamma^{A}_{Q[A,B]^{r}P})
 \\&(\cvl[A,B]^{r}P)
\\\langle \sigma_{A},\sigma_{B} \rangle_{r}^{-1}  := &([A,(B\evr )(\sigma^{-1}_{B}Q)(P\Theta^{A}_{B}Q)(\sigma_{A}[A,B]^{r}Q)]^{l}P)(\Gamma^{A}_{P[A,B]^{r}Q}P)
\\&(P[A,B]^{r}\cvr)
\end{align*}
define a suitable $P$-intertwining on $ [A,B]^{r}$ so that the endofunctor $[(A,\sigma_{A}),-]^{r}$ which sends a pair $(B,\sigma_{B})$ in $\ct{C}(P,Q)$ to $ \big([A,B]^{r},\langle \sigma_{A},\sigma_{B} \rangle_{r}  \big)$ is right adjoint to $-\tn (A,\sigma_{A})$. 

\begin{cor}\label{CRig} If $\ct{C}$ is left (right) rigid, then $\ct{C}(P,Q)$ is left (right) rigid.
\end{cor}
\begin{proof} The statement follows directly from Theorem \ref{TCld} and Corollary \ref{CRCld}, when restricted to the case of a left or right rigid monoidal category. Explicitly, if $(X,\sigma)$ is an object in $\ct{C}(P,Q)$, then the $P$-intertwinings induce on $\pr{X}= [X,\un]^{l}$ and $X^{\vee}=[X,1]^{r}$ (if they exist), are denote by $\sigma_{\pr{X}}$ and $\sigma_{X^{\vee}}$, respectively, and are given by 
\begin{align*}
\sigma_{\pr{X}}&= (\evl_{X}P \pr{X})(\pr{X}\sigma^{-1}\pr{X} )(\pr{X}P \cvl_{X})
\\ \sigma_{\pr{X}}^{-1}&=(\evr \pr{X}P)(P\evl_{X} Q \pr{X}P)(P\pr{X}\ov{\sigma} \pr{X}P)(P\pr{X}Q\cvl_{X}P)(P\pr{X} \cvr )  
\\ \sigma_{X^{\vee}} &=(PX^{\vee}\evl)(PX^{\vee}Q \evr_{X} P)(PX^{\vee}\ov{\sigma}^{-1} X^{\vee} P)(P\cvr_{X}Q X^{\vee} P) (\cvl X^{\vee} P)
\\\sigma_{X^{\vee}}^{-1}&= (X^{\vee} P \evr_{X} )(X^{\vee} \sigma X^{\vee} )( \cvr_{X} PX^{\vee})
\end{align*} 
providing the left and right duals of $(X,\sigma)$ in $\ct{C}(P,Q)$.\end{proof}
\begin{thm}\label{TpivCP} If $\varrho : \mathrm{id}_{\ct{C}}\rightarrow \prescript{\vee\vee}{}{(-)}$ is a pivotal structure on $\ct{C}$ and $P$ is fixed by $\prescript{\vee\vee}{}{(-)}$, $\pr{P}=Q$ and $\varrho_{P}=\mathrm{id}_{P}$, then $\ct{C}(P,Q)$ is pivotal and the forgetful functor $U$ preserves this pivotal structure.
\end{thm}
\begin{proof} In this case, the pivotal structure of $\ct{C}$ directly lifts to $\ct{C}(P,Q)$. Here, we demonstrate that $\varrho_{X}: (X,\sigma) \rightarrow \prescript{\vee\vee}{}{(X, \sigma)}$ commutes with the $P$-intertwinings for any object $(X,\sigma) $ of $\ct{C}(P,Q)$. Observe that by Corollary \ref{CRig}, $\prescript{\vee\vee}{}{(X, \sigma)}= (\prescript{\vee\vee}{}{X}, (\sigma_{\pr{X}})_{\prescript{\vee\vee}{}{X}})$ where 
\begin{align*}
(\sigma_{\pr{X}})_{\prescript{\vee\vee}{}{X}} &= (\evl_{\pr{X}}P (\prescript{\vee\vee}{}{X}))(\prescript{\vee\vee}{}{X}\sigma_{\pr{X}}^{-1}\prescript{\vee\vee}{}{X} )(\prescript{\vee\vee}{}{X}P \cvl_{\pr{X}})
\\&= (P (\prescript{\vee\vee}{}{X} ) \evl )( P (\prescript{\vee\vee}{}{\ov{\sigma}}) P) (\cvl \prescript{\vee\vee}{}{X} P )  
\end{align*}
Observe that in the above statement we are abusing notation and assuming that $\prescript{\vee\vee}{}{(-)}$ is strict monoidal whereas this is not necessarily the case and $\prescript{\vee\vee}{}{\ov{\sigma}}$ should denote a morphism from $\prescript{\vee\vee}{}{(Q\tn X)}$ to $\prescript{\vee\vee}{}{(X\tn Q)}$. However, this is not an issue since $\varrho_{X}$ is a monoidal isomorphism and commutes with the natural isomorphisms $\prescript{\vee\vee}{}{(Q\tn X)}\cong \prescript{\vee\vee}{}{Q}\tn \prescript{\vee\vee}{}{X}$. Since $\pr{\varrho_{X}} = \varrho_{\pr{X}}^{-1}$ holds, \cite{ng2007frobenius} Appendix A, we conclude that 
\begin{align*}
((\sigma_{\pr{X}})_{\prescript{\vee\vee}{}{X}} ) ( \varrho_{X} P) &= (P (\prescript{\vee\vee}{}{X} ) \evl )\big( P (\prescript{\vee\vee}{}{\ov{\sigma}})(\varrho_{Q}\tn \varrho_{X}) ( \varrho_{Q}^{-1} X )P\big) (\cvl X P ) 
\\&= (P (\prescript{\vee\vee}{}{X} ) \evl )(P(\varrho_{X} \tn \varrho_{Q}) P)( P \ov{\sigma}P)(P\varrho_{Q}^{-1} XP) (\cvl X P )
\\ &= (P\varrho_{X} ) ( PX \evl ) ( P\ov{\sigma}P ) (\cvl XP) = (P\varrho_{X} ) \sigma
\end{align*}
Hence $\varrho_{X}$ is morphism in $\ct{C}(P,Q)$ and lifts the pivotal structure of $\ct{C}$ trivially. 
\end{proof}
\begin{rmk}\label{RPiv} Notice that in the proof of Theorem \ref{TpivCP}, we only needed $\sigma$ for an arbitrary object $(X,\sigma)$ in $\ct{C}(P,Q)$ to commute with $\varrho_{P}$. Although this does not hold for arbitrary $P$-intertwinings, one could restrict to a subcategory of $\ct{C}(P,Q)$ where this additional condition holds. We will briefly discuss generalisations of this type in Section \ref{SExtension}.
\end{rmk}
Before concluding this section, we show that all colimits in $\ct{C}$ lift to $\ct{C}(P,Q)$. 
\begin{lemma}\label{LCol} If $\ct{C}$ is closed, the forgetful functor $U$ creates colimits. 
\end{lemma}
\begin{proof} Consider a diagram $\mathbb{D}:\ct{J}\rightarrow \ct{C}(P,Q)$ so that the diagram $U \mathbb{D}:\ct{J}\rightarrow \ct{C} $ has a colimit $A$ in $\ct{C}$ with a family of universal morphisms $\pi_{j}:U\mathbb{D}(j) \rightarrow A$ for objects $j$ in $ \ct{C}$. Since $\mathbb{D}:\ct{J}\rightarrow \ct{C}(P,Q)$ is a functor, we have a family of morphisms $\sigma_{j}:\mathbb{D}(j) \tn P \rightarrow P\tn \mathbb{D}(j)$ which are natural with respect to $\ct{J}$ and thereby form a natural transformation $\sigma : \mathbb{D} \tn P \Rightarrow P\tn \mathbb{D}$. Furthermore, because the category $\ct{C}$ is closed, the diagrams $U\mathbb{D}\tn P$ and $P\tn U\mathbb{D}$ admit colimits $A\tn P$ and $P\tn A$, respectively. By the universal property of $A\tn P$, there exists a unique morphism $\sigma_{A}$ such that $\sigma_{A}(\pi\tn P)= (P\tn \pi )\sigma$. Since $\sigma$ is invertible, it follows from the universal property of $P\tn A$ that there exists a unique morphism $\sigma_{A}^{-1}$ such that $\sigma^{-1}_{A}(P\tn \pi )=(\pi\tn P) \sigma^{-1}$. It follows that $\sigma_{A}$ and $\sigma_{A}^{-1}$ are inverses and similarly we conclude that the induced $Q$-intertwinings on $A$ are inverses. Hence, $(A,\sigma_{A})$ is an object of $\ct{C}(P,Q)$ and $\pi: \mathbb{D} \Rightarrow (A,\sigma_{A}) $ a cocone of the diagram. To demonstrate that $(A,\sigma_{A})$ is a colimit, consider another cocone $\kappa: \mathbb{D} \Rightarrow (B,\sigma_{B}) $. Since $A$ is a colimit of $U\mathbb{D}$, there exists a unique morphism $t:A\rightarrow B$ such that $U\kappa =t(U\pi )$. What remains to be shown is whether $t$ commutes with the $P$-intertwinings of $A$ and $B$ which follows from the universality of $A\tn P$ and the calculation below
\begin{align*}
(P\tn t )\sigma_{A}(U\pi\tn P)&= (P\tn t )(P\tn U\pi)= (P\tn U \kappa)
\\&=\sigma_{B}(U\kappa \tn P ) = \sigma_{B}(t \tn P )(U\pi\tn P )
\end{align*}
Hence $(P\tn t )\sigma_{A}= \sigma_{B}(t \tn P )$ and thereby, $(A,\sigma_{A})$ is a colimit of the original diagram $\mathbb{D}$. 
\end{proof}
\begin{cor} If $\ct{C}$ is a rigid abelian category, then $\ct{C}(P,Q)$ is rigid and abelian and the forgetful functor $U$ is exact. 
\end{cor} 
\begin{proof} Since in a rigid category $X \tn -$ and $-\tn X$ preserve limits as well as colimits, for arbitrary objects $X$ in $\ct{C}$, a symmetric proof to that of Lemma $\ref{LCol}$ demonstrates that $U$ creates limits. Furthermore, the additive structure of $\ct{C}$ lifts trivially and since $U$ creates all finite limits and colimits, $\ct{C}(P,Q)$ becomes abelian and $U$ exact. 
\end{proof}
We conclude this section with a small examples of what the category $\ct{C}(P,Q)$ looks like, for a well-known monoidal category. 
\begin{ex}\label{EGVec} Let $G$ be a finite group and  consider the monoidal category of finite dimensional $G$-graded vectorspaces $\mathrm{vec}_{G}$ with the usual monoidal structure, as described in Example 2.3.6 of \cite{etingof2016tensor} and denote its simple objects by $V_{g}$ where $g\in G$. Then for any $g\in G$, $P=V_{g}$ is pivotal and $Q=V_{g^{-1}}$ and the evaluation and coevaluation morphism are trivial identity morphisms of the ground field. Hence, the category $\ct{C}(V_{g},V_{g^{-1}})$ has pairs $(\oplus_{i=1}^{n}V_{h_{i}}, \sigma )$ as objects, where $n\in \mathbb{N}$, $h_{i}\in G$ and $ \sigma: \oplus_{i=1}^{n}V_{h_{i}g}\rightarrow \oplus_{i=1}^{n}V_{gh_{i}} $ is a $G$-graded isomorphism. Due to the trivial form of the the duality morphisms in $\mathrm{vec}_{G}$, for any such $\sigma$, $\ov{\sigma}$ and $\ov{\sigma}^{-1}$ will automatically be inverses. Note that for any object $(\oplus_{i=1}^{n}V_{h_{i}}, \sigma )$, the set $\lbrace h_{i} \mid 1\leq i\leq n\rbrace$ is a disjoint union of orbits of the conjugation action of $g$ on $G$.
\end{ex}
\section{Resulting Hopf Monads}\label{SMnd}
In this section, we assume that the category $\ct{C}$ is closed and has countable colimits. Thereby, $\tn$ commutes with colimits and the category of endofunctors $\mathrm{End}(\ct{C})$ also has countable colimits. Utilising this, we construct the Hopf monad whose Eilenberg-Moore category recovers $\ct{C}(P,Q)$. 

Observe that for a pair $(X,\sigma)$ in $\ct{C}(P,Q)$, we can view $\sigma$ and $\sigma^{-1}$ as certain \emph{actions} of the functors $Q\tn - \tn P$ and $P\tn - \tn Q$ on $X$: 
\begin{align*}
\xymatrix@C+43pt{Q\tn X \tn P\ar[r]^-{(\evl \tn X )(Q\tn \sigma)} & X  } \quad\quad \xymatrix@C+43pt{ P\tn X \tn Q\ar[r]^-{(X\tn \evr) (\sigma^{-1}\tn Q)}& X }
\end{align*}
Moreover, for any pair $(X,\sigma)$ in $\ct{C}(P,Q)$, we can translate the mentioned actions in terms of the induced $Q$-intertwinings, since $(X\tn \evl ) ( \ov{\sigma}\tn P) =(\evl \tn X )(Q\tn \sigma)$ and $(\evr \tn X)(P\tn \ov{\sigma}^{-1}) =(X\tn \evr) (\sigma^{-1}\tn Q)$.

Conversely, when provided with two morphisms $\alpha: Q\tn X \tn P\rightarrow X$ and $\beta :P\tn X \tn Q\rightarrow X$, we can recover right and left $P$-intertwinings as below: 
\begin{align*}
\xymatrix@C+55pt{ X \tn P\ar[r]^-{(P\tn \alpha)(\cvl \tn X \tn P )} & P\tn X } \quad \xymatrix@C+55pt{ P\tn X \ar[r]^-{(\beta\tn Q)(P\tn X\tn \cvr) }& X \tn P}
\end{align*}
If we want the induced $P$-interwtinings of $\alpha$ and $\beta$ to be inverses, we need the following equalities to hold: 
\begin{align}
\evr \tn X  = \beta(P\tn \alpha\tn Q)(P\tn Q \tn X\tn \cvl ) : P\tn Q \tn X\rightarrow X  \label{EqPQX}
 \\X\tn \evl =   \alpha ( Q\tn \beta \tn P ) ( \cvr \tn X\tn Q\tn P ) : X\tn Q\tn P\rightarrow X \label{EqXQP}
\end{align}
Similarly, $\alpha$ and $\beta$ induce $Q$-intertwinings, \ref{Eqovsig1} and \ref{Eqovsig2} which can be written as 
\begin{align*}
\xymatrix@C+55pt{ X \tn Q\ar[r]^-{(P\tn \beta)(\cvr \tn X \tn Q)} & X \tn Q} \quad \xymatrix@C+55pt{ Q\tn X \ar[r]^-{(\alpha\tn Q)(Q\tn X\tn \cvl) }& X \tn Q}
\end{align*}
In order for the induced $Q$-intertwinings to be inverses, we require the following equalities to hold:
\begin{align}
\evl \tn X  = \alpha(Q\tn \beta\tn P)(Q\tn P \tn X\tn \cvr ) : Q\tn P\tn X\rightarrow X \label{EqQPX}
 \\X\tn \evr =   \beta ( P\tn \alpha \tn Q ) ( \cvl\tn X\tn P\tn Q ) : X\tn P\tn Q\rightarrow X \label{EqXPQ}
\end{align}
With this view of $P$-intertwinings in mind, we construct the left adjoint functor to $U$. 

Define the endofunctors $F_{+},F_{-}:\ct{C} \rightarrow \ct{C}$ by 
$$F_{+}(X)=Q\tn X \tn P , \quad F_{-} (X)= P\tn X\tn Q  $$
Let the endofunctor $F^{\star}$ be defined as the coproduct 
$$F^{\star} = \coprod_{n\in \mathbb{N}\cup\lbrace 0\rbrace, (i_{1}, i_{2}, \dots, i_{n}) \in \lbrace -, + \rbrace^{n} } F_{i_{1}}F_{i_{2}} \cdots F_{i_{n}} $$
where the term $F_{i_{1}}F_{i_{2}} \cdots F_{i_{n}} $ at $n=0$, is just the identity functor $\id $. For arbitrary $ n\in \mathbb{N}$ and $ (i_{1}, i_{2}, \dots, i_{n}) \in \lbrace -, + \rbrace^{n} $, we denote $F_{i_{1}}F_{i_{2}} \cdots F_{i_{n}}$ by $F_{i_{1},i_{2},\dots , i_{n}}$ and the respective natural transformations $ F_{i_{1},i_{2},\dots , i_{n}}\Rightarrow F^{\star}$ by $\iota_{i_{1},i_{2},\dots , i_{n}}$. We denote the additional natural transformation $\id \Rightarrow F^{\star}$ by $\iota_{0}$. Hence, for any $F_{i_{1},i_{2},\dots , i_{n}}$ we have four parallel pairs: 
\begin{align} 
 \xymatrix@C+5cm{P\tn Q \tn F_{i_{1},i_{2},\dots , i_{n}} \ar@<0.5ex>[r]^-{\iota_{-,+, i_{1},i_{2},\dots , i_{n}}(P\tn Q \tn F_{i_{1},i_{2},\dots , i_{n}} \tn \cvl )}\ar@<-0.5ex>[r]_-{ \iota_{i_{1},i_{2},\dots , i_{n}}(\evr \tn F_{i_{1},i_{2},\dots , i_{n}})}& F^{\star}  }  \label{EqPPPQX}
\\ \xymatrix@C+5cm{ F_{i_{1},i_{2},\dots , i_{n}}\tn Q\tn P \ar@<0.5ex>[r]^-{\iota_{+,-, i_{1},i_{2},\dots , i_{n}}(\cvr \tn F_{i_{1},i_{2},\dots , i_{n}} \tn Q\tn P )}\ar@<-0.5ex>[r]_-{ \iota_{i_{1},i_{2},\dots , i_{n}}( F_{i_{1},i_{2},\dots , i_{n}}\tn \evl ) }& F^{\star}  }  \label{EqPPXQP}
\\  \xymatrix@C+5cm{Q\tn P \tn F_{i_{1},i_{2},\dots , i_{n}} \ar@<0.5ex>[r]^-{\iota_{+,-, i_{1},i_{2},\dots , i_{n}}(Q\tn P \tn F_{i_{1},i_{2},\dots , i_{n}} \tn \cvr )}\ar@<-0.5ex>[r]_-{ \iota_{i_{1},i_{2},\dots , i_{n}}(\evl \tn F_{i_{1},i_{2},\dots , i_{n}})}& F^{\star}  } \label{EqPPQPX}
\\ \xymatrix@C+5cm{ F_{i_{1},i_{2},\dots , i_{n}}\tn P\tn Q \ar@<0.5ex>[r]^-{\iota_{-,+, i_{1},i_{2},\dots , i_{n}}(\cvl \tn F_{i_{1},i_{2},\dots , i_{n}} \tn P\tn Q )}\ar@<-0.5ex>[r]_-{ \iota_{i_{1},i_{2},\dots , i_{n}}( F_{i_{1},i_{2},\dots , i_{n}}\tn \evr ) }& F^{\star}  }  \label{EqPPXPQ}
\end{align} 
Consider the diagram, in $\mathrm{End}(\ct{C})$, which the described parallel pairs create. We denote the colimit of this diagram by $T$, the unique natural transformation $F^{\star} \Rightarrow T$, by $\psi$, and the compositions $\psi \iota_{i_{1},i_{2},\dots , i_{n}}$ and $\psi\iota_{0}$, by $ \psi_{i_{1},i_{2},\dots , i_{n}}$ and $\psi_{0}$, respectively.

Since $\tn$ commutes with colimits, the family of morphisms
$$ \psi_{+, i_{1},i_{2},\dots , i_{n}} : Q\tn F_{i_{1},i_{2},\dots , i_{n}}\tn P  \rightarrow T $$ 
induce a unique morphism $ \alpha : Q\tn T \tn P \rightarrow T $ such that $ \alpha (Q\tn \psi_{i_{1},i_{2},\dots , i_{n}} \tn P)= \psi_{+, i_{1},i_{2},\dots , i_{n}}$. Similarly, the family of morphisms 
$$ \psi_{-, i_{1},i_{2},\dots , i_{n}} : P\tn F_{i_{1},i_{2},\dots , i_{n}}\tn Q  \rightarrow T $$ 
induce a morphism $ \beta : P\tn T \tn Q \rightarrow T $ such that $ \beta (P\tn \psi_{i_{1},i_{2},\dots , i_{n}} \tn Q)= \psi_{-, i_{1},i_{2},\dots , i_{n}}$. As mentioned at the start of the section, such actions $\alpha$ and $\beta$ provide us with the necessary $P$-intertwinings, but we must show that the induced $P$-intertwining belongs to $\ct{C}(P,Q)$.
\begin{lemma} For any object $X$ in $\ct{C}(P,Q)$, the pair 
$$\big( T(X), (P\tn \alpha_{X})(\cvl \tn T(X) \tn P)\big) $$ belongs to $\ct{C}(P,Q)$. 
\end{lemma} 
\begin{proof} As we demonstrated at the beginning of this section, we only need to check that equalities \ref{EqPQX}, \ref{EqXQP}, \ref{EqQPX} and \ref{EqXPQ} hold for the defined actions $\alpha_{X}$ and $\beta_{X}$. Consider equation \ref{EqPQX}. We observe that by construction 
\begin{align*}
(\evr&\tn  T) (P\tn Q \tn \psi_{i_{1},,i_{2},\dots , i_{n}})=\psi_{-,+, i_{1},i_{2},\dots , i_{n}}(P\tn Q \tn F_{i_{1},i_{2},\dots , i_{n}} \tn \cvr ) 
\\&= \beta(P\tn \psi_{+, i_{1},i_{2},\dots , i_{n}}\tn Q)(P\tn Q \tn F_{i_{1},i_{2},\dots , i_{n}} \tn \cvr ) 
\\&= \beta(P\tn \alpha \tn Q)(P\tn Q\tn \psi_{ i_{1},i_{2},\dots , i_{n}}\tn P\tn Q )(P\tn Q \tn F_{i_{1},i_{2},\dots , i_{n}} \tn \cvr )
\\&= \beta(P\tn \alpha \tn Q)(P\tn Q\tn T \tn \cvl )(P\tn Q\tn \psi_{ i_{1},i_{2},\dots , i_{n}} )
\end{align*}
and by the universal property of the functor $T$, we conclude that
$$ (\evr \tn  T)= \beta(P\tn \alpha \tn Q)(P\tn Q\tn T \tn \cvl )$$
It should be clear that \ref{EqXQP}, \ref{EqQPX} and \ref{EqXPQ} follow in a similar manner from the construction of the functor $T$, and we leave the details to the reader. \end{proof}
We denote the natural transformation $(P\tn \alpha)(\cvl \tn T \tn P ): T \tn P \Rightarrow P\tn T $ by $\sigma^{T}$. Hence, we define the functor $F:\ct{C}\rightarrow \ct{C}(P,Q)$ by  $F(X) = \big( T(X), \sigma^{T}_{X}\big)$ for objects $X$ of $\ct{C}$ and $F(f)=T(f) $ for morphisms $f$ of $\ct{C}$. By construction, $F$ is functorial. 
\begin{thm}\label{TAdj} The functor $F$ as defined above is left adjoint to $U$, making $F\dashv U$ a Hopf adjunction. 
\end{thm} 
\begin{proof} We provide the unit and counit of the adjunction explicitly and show they satisfy the necessary conditions. The unit of the adjunction was present in our construction as $\nu:=\psi_{0}:\id \Rightarrow UF=T$. For the counit, consider a pair $(X,\sigma )$ in $\ct{C}(P,Q)$ and denote its induced actions $ (\evl \tn X )(Q\tn \sigma)$ and $ (X\tn \evr) (\sigma^{-1}\tn Q)$ by $\alpha_{\sigma}: F_{+}(X)\rightarrow X$ and $ \beta_{\sigma}:F_{-}(X)\rightarrow X$, respectively. We can define $\theta_{ i_{1},i_{2},\dots , i_{n}}: F_{i_{1},i_{2},\dots , i_{n}} (X)\rightarrow X$, for arbitrary $ n\in \mathbb{N}$ and $ (i_{1}, i_{2}, \dots, i_{n}) \in \lbrace -, + \rbrace^{n} $, by iteratively applying $\alpha_{\sigma}$ and $\beta_{\sigma} $ so that $\theta_{+, i_{1},i_{2},\dots , i_{n}}= \alpha_{\sigma}(Q\tn\theta_{ i_{1},i_{2},\dots , i_{n}}\tn P)$ and $\theta_{-, i_{1},i_{2},\dots , i_{n}}= \beta_{\sigma}(P\tn \theta_{ i_{1},i_{2},\dots , i_{n}}\tn Q)$, where $\theta_{+}= \alpha_{\sigma}$ and $\theta_{-}= \beta_{\sigma}$. Together with $\theta_{0}=\mathrm{id}_{X}$, we have a family of morphisms from $F_{i_{1},i_{2},\dots , i_{n}}(X)$ to $X$, which must factorise through $F^{\star}(X)$. We denote the unique morphism $F^{\star}(X)\rightarrow X$ by $\theta^{\star}$ and observe that the family of morphisms described commute with the parallel pairs \ref{EqPPPQX}, \ref{EqPPXQP}, \ref{EqPPQPX} and \ref{EqPPXPQ} e.g. for the parallel pair \ref{EqPPPQX}: 
\begin{align*}
\theta^{\star} \iota&_{-,+, i_{1},i_{2},\dots , i_{n}}(P\tn Q \tn F_{i_{1},i_{2},\dots , i_{n}} \tn \cvr )
\\ =& \theta_{-,+, i_{1},i_{2},\dots , i_{n}} (P\tn Q \tn F_{i_{1},i_{2},\dots , i_{n}} \tn \cvr ) 
\\ =& \beta (P\tn \alpha \tn Q) ( P\tn Q\tn\theta_{ i_{1},i_{2},\dots , i_{n}} \tn Q\tn P )(P\tn Q \tn F_{i_{1},i_{2},\dots , i_{n}} \tn \cvr )
\\ =& \beta (P\tn \alpha \tn Q) (P\tn Q \tn X\tn \cvr )( P\tn Q\tn\theta_{ i_{1},i_{2},\dots , i_{n}} ) 
\\ =&\beta (P\tn (\evl \tn X) (Q\tn \sigma) \tn Q) (P\tn Q \tn X\tn \cvr )( P\tn Q\tn\theta_{ i_{1},i_{2},\dots , i_{n}} )
\\= &(X\tn \evr )(\sigma^{-1} \tn Q )(P\tn\ov{\sigma})( P\tn Q\tn\theta_{ i_{1},i_{2},\dots , i_{n}} )
\\ =& (\evr \tn X )( P\tn Q\tn\theta_{ i_{1},i_{2},\dots , i_{n}} ) = \theta^{\star} (\evr \tn \iota_{i_{1},i_{2},\dots , i_{n}})
\end{align*} 
Similar calculations follow for parallel pairs \ref{EqPPXQP}, \ref{EqPPQPX} and \ref{EqPPXPQ} from the properties of $\sigma$. Hence, by the universal property of $T(X)$, we conclude that there exists a unique morphism $\theta_{(X,\sigma )} :T(X)\rightarrow X$ such that $\theta_{X} \psi_{ i_{1},i_{2},\dots , i_{n}} = \theta_{ i_{1},i_{2},\dots , i_{n}}$. In fact, $\theta $ is a morphism between $(T(X),\sigma^{T}_{X}) $ and $(X,\sigma)$ in $\ct{C}(P,Q)$ i.e. $(P\tn \theta_{X} )\sigma^{T}_{X}= \sigma(\theta_{X}\tn P)$ holds: this is equivalent to $\alpha_{\sigma}(Q\tn \theta_{X} \tn P)= \theta_{X} \alpha_{X}$ which holds by definition of $\theta_{X}$. By universality of $T$, $\theta$ is natural and we have described a natural transformation $\theta : FU \Rightarrow \mathrm{id}_{\ct{C}(P,Q)}$. The triangle identities for the unit and counit $\nu$ and $\theta$ follow trivially by the universal property of $\theta$ since $\theta_{(T(X),\sigma^{T}_{X})}\nu_{X} =\mathrm{id}_{X}$ by definition.\end{proof}

\begin{cor} The adjunction $F\dashv U$ is monadic and the monad $(T, U\theta_{F} , \nu )$ is a Hopf monad. 
\end{cor}
\begin{proof} By Lemma \ref{LCol} and Beck's Theorem \ref{TBeck}, the adjunction is monadic. By Theorem \ref{TCld}, the induced monad $(T, \epsilon , \eta )$ is a Hopf monad. \end{proof} 
At this point we would like to take a step back and look at the particular structure of $T$ as a bimonad. At the beginning of the section, we described how a $P$-intertwining $\sigma$ on an object $X$ is equivalent to a pair of suitable actions $\alpha_{\sigma}$ and $\beta_{\sigma}$ on $X$. In the proof of Theorem \ref{TAdj}, we showed that for any object $(X,\sigma)$ in $\ct{C}(P,Q)$, there exists a unique morphism $\theta_{(X,\sigma)}:T(X)\rightarrow X$ so that  $\theta\psi_{ i_{1},i_{2},\dots , i_{n}} = \theta_{ i_{1},i_{2},\dots , i_{n}}$, where $\theta_{ i_{1},i_{2},\dots , i_{n}}$ are just the iterative applications of $\alpha_{\sigma}$ and $\beta_{\sigma}$. In particular, for any object $X$ of $\ct{C}$, $T(X)$ has naturally suitable actions $\alpha_{X}$ and $\beta_{X}$ which satisfy $ \alpha_{X} (Q\tn (\psi_{i_{1},i_{2},\dots , i_{n}})_{X} \tn P)= (\psi_{+, i_{1},i_{2},\dots , i_{n}})_{X}$ and $ \alpha_{X} (P\tn (\psi_{i_{1},i_{2},\dots , i_{n}})_{X} \tn Q)= (\psi_{-, i_{1},i_{2},\dots , i_{n}})_{X}$. Hence, for the pair $F(X)=(T(X),\sigma^{T}_{X})$, $\theta_{F(X)} :TT(X) \rightarrow T(X)$ is the unique morphism such that 
$$ \theta_{F(X)} (\psi_{i_{1},\dots , i_{n}})_{T(X)} F_{ i_{1},\dots , i_{n}}\big( (\psi_{j_{1},\dots , j_{m}})_{X}\big)= (\psi_{i_{1},\dots , i_{n},j_{1},\dots , j_{m}})_{X}$$
for arbitrary non-negative integers $n,m$ and $i_{1},i_{2},\dots , i_{n}, j_{1},j_{2},\dots , j_{m}\in \lbrace +,-\rbrace$. As previously mentioned $\nu=\psi_{0}: \id  \rightarrow T$. The comonoidal structure of $T$ arises directly from the monoidal structure of $\ct{C}(P,Q)$. Observe that for pairs $(X,\sigma)$ and $(Y,\tau)$ , the induced action $\alpha_{\sigma\tn\tau}$ on $A\tn B$ is the composition $( \alpha_{\sigma}\tn \beta_{\sigma} )(P\tn X \tn \cvr \tn Y\tn Q)$. With this in mind, we observe that the comonoidal structure of $T$, $T_{2}: T(-\tn -)\rightarrow T(-)\tn T(-)$, is the unique morphism such that 
$$T_{2}\psi_{i_{1},\dots , i_{n}}=\big( \psi_{i_{1},\dots , i_{n}}\tn \psi_{i_{1},\dots , i_{n}}\big) F_{i_{1},\dots , i_{n}} ( -\tn \cvl_{i_{1},\dots , i_{n}}\tn -) $$
where $\cvl_{i_{1},\dots , i_{n}}: \un \rightarrow F_{-i_{n},\dots , -i_{1}}(\un)$ are iteratively defined by $\cvl_{+,i_{1},i_{2},\dots , i_{n}} = F_{-i_{n},\dots , -i_{1}}( \cvl )$ and $\cvl_{-,i_{1},i_{2},\dots , i_{n}} =F_{-i_{n},\dots , -i_{1}}( \cvr )$
where $\cvl_{+} =\cvl  $ and $\cvl_{-} =\cvr $. Recall that the $P$-intertwining making $\un $ the unit of the monoidal structure in $\ct{C}(P,Q)$ is simply the identity morphism $\mathrm{id}_{P}$ and its induced actions are $\alpha_{\mathrm{id}_{P}}= \evr $ and $\beta_{\mathrm{id}_{P}} = \evl$. Hence, morphism $T_{0}: T(\un ) \rightarrow \un $ is the unique morphism so that $T_{0} \psi_{i_{1},i_{2},\dots , i_{n}}= \evl_{i_{1},i_{2},\dots , i_{n}}$, where $\evl_{i_{1},\dots , i_{n}}: F_{i_{1},\dots , i_{n}} (\un)\rightarrow \un$ is defined iteratively by $\evl_{+,i_{1},\dots , i_{n}} = \evl F_{+}(\evl_{i_{1},i_{2},\dots , i_{n}})$ and $\evl_{-,i_{1},\dots , i_{n}} = \evr F_{-}(\evl_{i_{1},i_{2},\dots , i_{n}})$, with $\evl_{+}= \evl$ and $\evl_{-}=\evr$. 
\begin{thm}\label{ThmTAug} The Hopf monad $T$ is augmented if and only if there exist a braidings $\lambda : P\tn \id  \Rightarrow  \id  \tn P $ and $\chi : Q\tn \id  \Rightarrow  \id  \tn Q $  so that $(P,\lambda)$ and $(Q,\chi )$ are objects in $Z(\ct{C})$ and $\cvl,\evl,\cvr$ and $\evr$ are morphisms in $Z(\ct{C})$, making $(P,\lambda)$ and $(Q,\chi )$ a pivotal pair in $Z(\ct{C})$ . 
\end{thm}
\begin{proof} ($\Rightarrow$) Assume $T$ is augmented and there exists a Hopf monad morphism $\xi: T\Rightarrow \id $. Hence $\xi$ satisfies $\xi \nu =\id  $, $\xi\theta = \xi T(\xi)$, $(\xi \tn \xi )T_{2}=\xi _{-\tn - }$ and $ \xi_{\un}= T_{0}$. We consider the natural transformation $\lambda := (\xi \psi_{-} \tn P)( P\tn \id  \tn \cvr ) $. We now demonstrate that $\lambda$ is a braiding as required. First observe that $\lambda$ is invertible and $ \lambda ^{-1}:=(P\tn \xi \psi_{+} )( \cvl \tn \id  \tn P ) $ provides its inverse: 
\begin{align*}
\lambda \lambda^{-1} &=(\xi \psi_{-} \tn P)( P\tn \id  \tn \cvr )(P\tn \xi \psi_{+} )( \cvl \tn \id  \tn P )
\\  &= (\xi \psi_{-} \tn P)(F_{-}(\xi \psi_{+} )\tn P )( \cvl \tn \id  \tn P\tn \cvr )
\\  &= \big(\xi T(\xi) (\psi_{-})_{T} F_{-}( \psi_{+} )\tn P \big)( \cvl \tn \id  \tn P\tn \cvr )
\\  &= \big(\xi \theta (\psi_{-})_{T} F_{-}( \psi_{+} )\tn P \big)( \cvl \tn \id  \tn P\tn \cvr )
\\  &= (\xi \psi_{-,+}\tn P )( \cvl \tn \id  \tn P\tn \cvr )
\\ &= (\xi \psi_{0}\evr \tn P )(\id  \tn P\tn \cvr )= \id 
\end{align*}
The calculation showing $\lambda^{-1}\lambda = \id $ is completely symmetric and left to the reader. Observe that the braiding conditions follow from the properties of $\xi$. We can directly deduce that $\lambda_{\un} = \mathrm{id}_{P} $ since
\begin{align*}
(\xi (\psi_{-} )_{\un}\tn P)( P\tn \un\tn \cvr ) &= (T_{0} (\psi_{-} )_{\un}\tn P)( P\tn \un\tn \cvr ) 
\\ & = (\evr \tn P)( P\tn \un\tn \cvr ) = \mathrm{id}_{P}
\end{align*}
and $\lambda_{X\tn Y} = (X\tn \lambda_{Y})(\lambda_{X} \tn Y ) $ holds for any arbitrary pair of objects $X$ and $Y$ in $\ct{C}$ since 
\begin{align*}
\lambda&_{X \tn Y} = (\xi_{X\tn Y} (\psi_{-})_{X\tn Y}  \tn P)( P\tn X\tn Y \tn \cvr )
\\&=\big( (\xi_{X}\tn \xi_{Y}) T_{2}(X,Y) (\psi_{-})_{X\tn Y}  \tn P\big)( P\tn X\tn Y \tn \cvr )
\\ &=\big( \xi_{X}(\psi_{-})_{X}\tn \xi_{Y}(\psi_{-})_{Y}  \tn P\big)( P\tn X\tn \cvr \tn Y \tn \cvr )
\\ &= (X\tn \xi_{Y}(\psi_{-})_{Y} )(X\tn P\tn  Y \tn \cvr ) ( \xi_{X}(\psi_{-})_{X}\tn Y)( P\tn X\tn \cvr \tn Y )
\\ &= (X\tn \lambda_{Y})(\lambda_{X} \tn Y ) 
\end{align*}
Hence, $(P,\lambda )$ is an object in the center of $\ct{C}$. In the same manner, one can deduce that $\chi := (\xi \psi_{+} \tn Q)( Q\tn \id  \tn \cvl ) $ is a braiding with $ \chi ^{-1}:=(Q\tn \xi \psi_{-} )( \cvr \tn \id  \tn Q ) $ as its inverse. What remains to be checked is whether $\cvl,\evl,\cvr$ and $\evr$ are morphisms in $Z(\ct{C})$ and commute with the braidings of $\un$, $P\tn Q$ and $Q\tn P$. For $\evl$ we must demonstrate that $\evl \tn \id = (\id  \tn \evl  )(\chi \tn P)(Q\tn \lambda)$ which follows by considering the parallel pair \ref{EqPPQPX}:
\begin{align*}
(\id  \tn &\evl  )(\chi \tn P)(Q\tn \lambda) =(\id  \tn \evl  )(\chi \tn P) (Q\tn \xi \psi_{-} \tn P)
\\ &(Q\tn P\tn \id  \tn \cvr )
\\ =&(\id  \tn \evl  )(\xi \psi_{+} \tn Q\tn P)( Q\tn  \xi \psi_{-} \tn \cvl\tn P  ) (Q\tn P\tn \id  \tn \cvr )
\\ =& \xi T(\xi)(\psi_{+})_{T} F_{+}(\psi_{-})  (Q\tn P\tn \id  \tn \cvr )
\\ =& \xi \theta  (\psi_{+})_{T} F_{+}(\psi_{-})  (Q\tn P\tn \id  \tn \cvr )
\\ =&\xi \psi_{+,-}  (Q\tn P\tn \id  \tn \cvr )= \xi \psi_{0} (\evl \tn \id  ) =\evl \tn \id 
\end{align*}
For $\cvl$ we must show that $(\lambda \tn Q)(P\tn \chi) (\cvl\tn \id  )=\id \tn \cvl  $ which follows by considering the parallel pair \ref{EqPPXPQ}:
\begin{align*}
(\lambda \tn Q&)(P\tn \chi) (\cvl\tn \id  )= (\xi \psi_{-} \tn P\tn Q)( P\tn \id  \tn \cvr  \tn Q)
\\ &(P\tn \xi \psi_{+} \tn Q) (\cvl\tn \id \tn \cvl  )
\\ =&(\xi T(\xi) (\psi_{-})_{T}F_{-}(\psi_{+}) \tn P\tn Q)(P\tn F_{+} \tn \cvr  \tn Q)(\cvl\tn \id \tn \cvl  )
\\ =&(\xi \theta (\psi_{-})_{T}F_{-}(\psi_{+}) \tn P\tn Q)(P\tn F_{+} \tn \cvr  \tn Q)(\cvl\tn \id \tn \cvl  )
\\=&(\xi \psi_{-,+} \tn P\tn Q)(\cvl\tn \id \tn P \tn \cvr  \tn Q) ( \id \tn \cvl  )
\\ =&(\xi \psi_{0} \tn P\tn Q)( \id \tn \cvl  )=  \id \tn \cvl  
\end{align*}
In a symmetric fashion, by looking at \ref{EqPPPQX} and \ref{EqPPXQP} one can show that $\cvr$ and $\evr $ are also morphisms in $Z(\ct{C})$.

($\Leftarrow$) Assume there exist braidings $\lambda : P\tn \id  \Rightarrow  \id  \tn P $ and $\chi : Q\tn \id  \Rightarrow  \id  \tn Q $ making $(P,\lambda)$ and $(Q,\chi )$ objects in $Z(\ct{C})$, such that $\cvl,\evl,\cvr$ and $\evr$ are morphisms in $Z(\ct{C})$. We can iteratively define the natural transformations $\xi_{i_{1}, \dots i_{n}}: F_{i_{1}, \dots i_{n}} \Rightarrow \id $ by 
\begin{align*}
\xi_{+,i_{1}, \dots i_{n}}&=  ( \id \tn \evl) (\chi \tn P)F_{+}(\xi_{i_{1}, \dots i_{n}})
\\ \xi_{-,i_{1}, \dots i_{n}}&=  (\id \tn \evr ) ( \lambda\tn Q)F_{-}(\xi_{i_{1}, \dots i_{n}})
\end{align*}
where $\xi_{+}= ( \id \tn \evl)(\chi\tn P)$ and $\xi_{-}=(\id \tn \evr ) ( \lambda\tn Q)$ and $\xi_{0}=\id $. Since $\evl$ and $\evr$ commute with the braidings, then $\xi_{+}= (\evl \tn  \id )(Q\tn \lambda^{-1})$ and $\xi_{-}=(\evr\tn \id  ) (P\tn \chi^{-1})$. It is straightforward to check that $\xi_{i_{1}, \dots i_{n}}$ commute with the parallel pairs \ref{EqPPPQX}, \ref{EqPPXQP}, \ref{EqPPQPX} and \ref{EqPPXPQ}, and therefore induce a unique morphism $\xi :T \rightarrow \id $ e.g. for parallel pair \ref{EqPPPQX}:
\begin{align*}
\xi_{-,+,i_{1}, \dots i_{n}}&(P\tn Q\tn F_{i_{1}, \dots i_{n}}\tn \cvl ) = (\id \tn \evr)(\lambda \tn Q) ( P\tn \id \tn \evl\tn Q)
\\& (P\tn \chi\tn P\tn Q)F_{-,+}(\xi_{i_{1}, \dots i_{n}}) (P\tn Q\tn F_{i_{1}, \dots i_{n}}\tn \cvl )
\\ = &(\id \tn \evr)(\lambda\tn Q)(P\tn \chi)(P\tn Q\tn\xi_{i_{1}, \dots i_{n}} )
\\ = &(\evr\tn \id  )(P\tn Q\tn\xi_{i_{1}, \dots i_{n}} )= \xi_{i_{1}, \dots i_{n}} (\evr\tn \id  )
\end{align*}
Similar arguments follow for $\xi_{i_{1}, \dots i_{n}}$ commuting with parallel pairs \ref{EqPPXQP}, \ref{EqPPQPX} and \ref{EqPPXPQ}. Observe that by definition $\xi \nu =\id $ and $\xi_{\un}=T_{0}$ since $(\xi_{i_{1}, \dots i_{n}})_{\un} = (\evl_{i_{1}, \dots i_{n}})$. Furthermore, 
\begin{align*}
\xi\theta&(\psi_{i_{1},\dots , i_{n}})_{T} F_{ i_{1},\dots , i_{n}}( \psi_{j_{1},\dots , j_{m}})= \xi\psi_{i_{1},\dots , i_{n},j_{1},\dots , j_{m}} = \xi_{i_{1},\dots , i_{n},j_{1},\dots , j_{m}}
\\&=\xi_{i_{1},\dots , i_{n}} F_{i_{1},\dots , i_{n}}(\xi_{j_{1},\dots , j_{m}}) = \xi\psi_{i_{1},\dots , i_{n}} F_{i_{1},\dots , i_{n}}(\xi\psi_{j_{1},\dots , j_{m}})
\\ &= \xi T(\xi)(\psi_{i_{1},\dots , i_{n}})_{T} F_{ i_{1},\dots , i_{n}}( \psi_{j_{1},\dots , j_{m}})
 \\(\xi \tn \xi )T_{2}&(\psi_{i_{1}, \dots i_{n}})_{-\tn - }=  \big( \xi\psi_{i_{1},\dots , i_{n}}\tn \xi\psi_{i_{1},\dots , i_{n}}\big) F_{i_{1},\dots , i_{n}} ( -\tn \cvl_{i_{1},\dots , i_{n}}\tn -)
\\&=\big( \xi_{i_{1},\dots , i_{n}}\tn \xi_{i_{1},\dots , i_{n}}\big) F_{i_{1},\dots , i_{n}} ( -\tn \cvl_{i_{1},\dots , i_{n}}\tn -)=(\xi_{i_{1}, \dots i_{n}})_{-\tn - }
\end{align*}
and from the universal properties of $TT$ and $T(-\tn - )$, we conclude that $\xi \theta = \xi T(\xi)$ and $(\xi \tn \xi )T_{2}= \xi_{-\tn - }$ and thereby, $\xi$ is a bimonad morphism.
\end{proof} 
Recall that an augmentation on a Hopf monad is equivalent to a central Hopf algebra structure on $T(\un)$. To be more precise, $T(\un )$ together with $T_{2}(\un, \un)  :T(\un)\rightarrow T(\un) \tn T(\un ) $ and $T_{0} : T(\un ) \rightarrow \un$ form a comonoid in $\ct{C}$. Additionally, as we recalled in Theorem \ref{ThmAug}, if $(P,\lambda )$ and $(Q, \chi )$ is a pivotal pair in the center of $\ct{C}$, then we have a monoid structure on $T(\un )$, $m: T(\un) \tn T(\un )\rightarrow T(\un )$, $\psi_{0}(\un) : \un \rightarrow T(\un)$ where $m$ is the unique morphism satisfying  
\begin{align*}
m ( \psi_{i_{1},\dots , i_{n}}(\un) &\tn \psi_{j_{1},\dots , j_{m}}(\un) ) =\psi_{i_{1},\dots , i_{n},j_{1},\dots , j_{m}} (\un) 
\\ &\big( Q_{i_{1}}\tn\dots \tn Q_{i_{n}}\tn \lambda_{i_{1},\dots , i_{n}} ( Q_{j_{1}}\tn \dots \tn Q_{ j_{m}} \tn P_{j_{1}}\tn \dots \tn P_{j_{n}} )\big)
\end{align*}
where we denote $P_{+}=P$, $P_{-}=Q$, $Q_{+}=Q$ and $Q_{-}=P$ so that $F_{\pm}(-)= Q_{\pm}\tn -\tn P_{\pm}$ and $\lambda_{i_{1},\dots , i_{n}} : P_{i_{1}}\tn\dots \tn P_{i_{n}} \tn \id  \rightarrow \id  \tn  P_{i_{1}}\tn\dots \tn P_{i_{n}}$ is the induced braiding on $P_{i_{1}}\tn\dots \tn P_{i_{n}}$, where $\lambda_{+}=\lambda$, $\lambda_{-}= \chi$ and $\lambda_{i_{1},\dots , i_{n}}= (\lambda_{i_{1}} \tn P_{i_{2}}\tn\dots \tn P_{ i_{n}} ) (P_{i_{1}}\tn \lambda_{i_{2},\dots , i_{n}})$. Observe that $T(\un )$ has an induced braiding $\varsigma : T(\un ) \tn \id  \Rightarrow \id \tn T(\un )$ satisfying 
$$\varsigma ( \psi_{i_{1},\dots , i_{n}}(\un)  \tn \id  ) = (\id \tn \psi_{i_{1},\dots , i_{n}}(\un)) \lambda_{-i_{1},\dots , -i_{n},i_{n},\dots , i_{1}}  $$ 
It follows that $ (T(\un) ,\varsigma ) $ is an object in the center of $\ct{C}$ and together with $m$, $\psi_{0}(\un)$, $ T_{2}(\un,\un )$, $T_{0}$ forms as central bialgebra. Moreover, we have an induced isomorphism of bimonads $\Upsilon : T(-)\Rightarrow T(\un) \tn - $ defined as the unique morphism satisfying 
$$\Upsilon \psi_{i_{1},\dots , i_{n}} =( \psi_{i_{1},\dots , i_{n}}(\un) \tn \id ) (Q_{i_{1}}\tn\dots \tn Q_{i_{n}}\tn \lambda^{-1}_{i_{n},\dots , i_{1}} )$$ 
Additionally, $T(\un )$ becomes a central Hopf algebra and its antipode $S:T(\un) \rightarrow T(\un)$ can be recovered as the unique morphism satisfying 
\begin{align*}S \psi_{i_{1},\dots , i_{n}} = &(\evl_{i_{1},\dots , i_{n}}\tn \psi_{-i_{n},\dots , -i_{1}}(\un))(Q_{i_{1}}\tn\dots \tn Q_{i_{n}} \tn \lambda^{-1}_{i_{n},\dots , i_{1}}( F_{-i_{n},\dots , -i_{1}}) ) 
\\ &(F_{i_{1},\dots , i_{n}} (\un ) \tn \cvl_{-i_{n},\dots , -i_{1}})
\end{align*}
with its inverse $S^{-1}: T(\un) \rightarrow T(\un) $ defined as the unique morphism satisfying
\begin{align*} S^{-1} &\psi_{i_{1},\dots , i_{n}} = (\evl_{i_{1},\dots , i_{n}}\tn \psi_{-i_{n},\dots , -i_{1}}(\un))
\\ &(Q_{i_{1}}\tn\dots \tn Q_{i_{n}} \tn \lambda^{-1}_{i_{n},\dots , i_{1}}( F_{-i_{n},\dots , -i_{1}}) ) (\cvl_{-i_{n},\dots , -i_{1}}\tn F_{i_{1},\dots , i_{n}} (\un )  )
\end{align*} 
We now review the structure of the constructed Hopf monad on some familiar categories. 
\begin{ex}\label{EBraidHpf} If $\ct{C}$ is braided with braiding $\Psi$, then $(P,\Psi_{P,-})$ and $(Q,\Psi_{Q,-})$ naturally form a pivotal pair in $Z(\ct{C})$ and by Theorem \ref{ThmAug}, $T$ is augmented. In particular, $T\cong T(1)\tn - $ where $T(1)$ is in fact a braided Hopf algebra in $\ct{C}$ since the induce braiding on $T(\un)$ will naturally coincide with the braiding $\Psi_{T(\un),-}$.
\end{ex} 
\begin{ex}\label{EHopfAlg} Let $\field$ be an arbitrary field and consider its symmetric monoidal category of vector spaces $(\mathrm{Vec}, \tn ,\field )$, where $\tn$ denotes the tensor product over the field. Any finite dimensional vectorspace is dualizable and pivotal with its dual vectorspace having the same dimension and the trivial evaluation and coevaluation providing the duality morphism for both sides. Since the category is symmetric, by Theorem \ref{ThmAug}, any Hopf monad constructed as above is augmented. Hence, the monad should arise from a Hopf algebra. Explicitly, for an $n$-dimensional vectorspace, the monad $F^{\star}$ is isomorphic to $B\tn - $, where $B$ is the free algebra $\field \langle \omx{i}{f}{j},\omx{i}{e}{j}\mid 1\leq i,j \leq n \rangle$ and generators $\omx{i}{f}{j} $ and $\omx{i}{e}{j}$ correspond to the bases of $P\tn Q$ and $Q\tn P$, respectively. Consequently, the monad $T$ is isomorphic to the induced monad, $H\tn - $, where $H$ is the quotient of the algebra $B$ by relations 
$$\sum_{j=1}^{n} \omx{j}{f}{i} .\omx{j}{e}{k} = \sum_{j=1}^{n} \omx{i}{f}{j} .\omx{k}{e}{j} = \sum_{j=1}^{n} \omx{j}{e}{i} .\omx{j}{f}{k}  = \sum_{j=1}^{n} \omx{i}{e}{j} .\omx{k}{f}{j} = \delta_{i,k}  $$
for all $1\leq i,k\leq n$. The coproduct, counit and antipode, $\Delta, \epsilon, S$ of the Hopf algebra $H$ are defined as 
\begin{align*}
\Delta ( \omx{i}{f}{k}) =   \sum_{j=1}^{n} \omx{i}{f}{j} \tn \omx{j}{f}{k}, \quad \epsilon (\omx{i}{f}{k})=\delta_{i,k}, \quad S(\omx{i}{f}{k}) = \omx{k}{e}{i}
\\\Delta ( \omx{i}{e}{k}) =   \sum_{j=1}^{n} \omx{i}{e}{j} \tn \omx{j}{e}{k}, \quad \epsilon (\omx{i}{e}{k})=\delta_{i,k}, \quad S(\omx{i}{e}{k}) = \omx{k}{f}{i}
\end{align*} 
for all $1\leq i,k\leq n$.
\end{ex} 
\begin{rmk}\label{RNGL} The Hopf algebra constructed in the above example can be viewed as a suitable quotient of the \emph{free Matrix Hopf algebra} $\ct{NGL}(n)$ discussed in \cite{skoda2003localizations}. The free matrix bialgebra of rank $n^{2}$ is exactly the free algebra $\field \langle \omx{i}{f}{j}\mid 1\leq i,j\leq n \rangle$ with the coproduct and counit defined as in Example \ref{EHopfAlg}. The Hopf envelope of this bialgebra as defined by Manin \cite{manin1988quantum}, would be the quotient of $\coprod_{l\in \mathbb{N}} B_{i}$ where $B_{l}=\field \langle \omx{i}{f^{l}}{j}\mid 1\leq i,j\leq n \rangle$ by the relations $\sum_{j=1}^{n} \omx{i}{f^{l}}{j} .\omx{k}{e^{l+1}}{j} = \sum_{j=1}^{n} \omx{j}{f^{l+1}}{i} .\omx{j}{f^{l}}{k}  = \delta_{i,k}$ for all $l\in \mathbb{N}$ and the antipode is defined as the shift $S(\omx{i}{f^{l}}{j}) = \omx{j}{f^{l+1}}{i} $. The Hopf algebra of Example \ref{EHopfAlg} is just the quotient of this Hopf algebra by requiring the $B_{l}$ components to be equal for odd $l$ and similairly for even $l$, so that $S^{2}= \mathrm{id}$.
\end{rmk}
\begin{ex}\label{EHopfMany} Notice that any $n$-dimensional $\field$-vectorspace, $P$, in fact belongs to a $GL(n,\field )$ moduli of pivotal pairs. Let $(P,Q)$ denote a pivotal pair in $\mathrm{Vec}$, with $\cvl, \evl, \cvr, \evr$ representing the relevant duality morphisms. We can always pick a basis $\lbrace v_{i}\rbrace_{i=1}^{n}$ for $P$ and change the basis of $Q$, to $\lbrace w_{i}\rbrace_{i=1}^{n}$, so that $\evl (w_{i}, v_{j})=\delta_{i,j} $ and $\cvl (1)= \sum_{i=1}^{n} v_{i}\tn w_{i}$. Consequently, $\cvr$ must be of the form $\cvr (1)= \sum_{i,j=1}^{n} q_{ij} w_{i}\tn v_{j}$, where $\mathfrak{Q}= (q_{ij})_{1\leq i,j\leq n}$ forms an invertible matrix and $\evr ( v_{i} \tn w_{j})= p_{ij}$, where $\mathfrak{Q}^{-1}= (p_{ij})_{1\leq i,j\leq n}$. Hence, we can associate a Hopf algebra to each invertible matrix $\mathfrak{Q}$, as a quotient of the free algebra $\field \langle \omx{i}{f}{j},\omx{i}{e}{j}\mid 1\leq i,j \leq n \rangle$ by relations 
$$\sum_{j=1}^{n} \omx{j}{f}{k} .\omx{j}{e}{i} = \sum_{j=1}^{n} \omx{i}{f}{j} .\omx{k}{e}{j} =p_{ik},\quad \sum_{j,l=1}^{n} \omx{j}{e}{i} .\omx{l}{f}{k} q_{jl} = \sum_{j,l=1}^{n}\omx{i}{e}{j} .\omx{k}{f}{l}  q_{lj}= \delta_{i,k}  $$
for all $1\leq i,k\leq n$. The coproduct and counit, $\Delta, \epsilon$, of the Hopf algebra are defined as in Example \ref{EHopfAlg}, for $\omx{i}{e}{k}$, and extended by 
\begin{align*}
\Delta ( \omx{i}{f}{k}) =   \sum_{j,l=1}^{n} \omx{i}{f}{j} \tn \omx{l}{f}{k}q_{jl}, \quad \epsilon (\omx{i}{f}{k})=p_{ik}
\\ S(\omx{i}{e}{k}) =\sum_{l=1}^{n}\omx{k}{f}{l}q_{li}, \quad S(\omx{i}{f}{k}) = \sum_{l=1}^{n}\omx{k}{e}{l}p_{il}
\end{align*} 
for all $1\leq i,k\leq n$.
\end{ex}

\begin{ex}\label{EHpfAlebroid}[Theorem 4.11 \cite{ghobadi2020hopf}] If $A$ is a $\field$-algebra and $P$ a pivotal object in the category of $A$-bimodules, $\bim$, the arising Hopf monad was constructed in \cite{ghobadi2020hopf}. As proven in \cite{szlachanyi2003monoidal}, additive bimonads and Hopf monads on $\bim$ which admit a right adjoint correspond to left bialgebroids and Hopf algebroids over $A$, in the sense of Schauenburg \cite{schauenburg2000algebras}. We adapt the notation of \cite{bohm2018hopf} to describe the Hopf algebroid in question and refer the reader to \cite{ghobadi2020hopf} where the construction is described in full detail. Consider the $A\tn_{\field} A^{op}$-bimodule structure induced on $Q\tn_{\field}P $ (resp. $P\tn_{\field} Q$) where we regard $Q$ (resp. $P$) as an $A$-bimodule and $P$ (resp. $Q$) as an $A^{op}$-bimodule, where $A^{op}$ denotes the opposite algebra to $A$. We denote arbitrary elements of $Q\tn_{\field}P $ and $P\tn_{\field} Q$ by $(q,p)$ and $(p,q)$, respectively, and we denote elements of $A^{op}$ in $H$ by a line over head i.e. $a\in A\ \subset H$ and $\ov{a}\in A^{op}\subset H$.. We define the Hopf algebroid $H$ as the quotient of the free $A\tn_{\field} A^{op}$-algebra $T_{A\tn_{\field} A^{op}} ( Q\tn_{\field}P\oplus P\tn_{\field} Q )$ by relations 
\begin{align*}
  \sum_{i=1}^{n}(\omega_{i},q) (x_{i},p )&=\ov{\evr(p\otimes q)}  \quad \sum_{j=1}^{m}(y_{j},p) (\rho_{j},q)=\ov{\evl(q\tn  p)} 
\\   \sum_{j=1}^{m} (q,\rho_{j})(p, y_{j})&=\evl(q\otimes p) \quad \sum_{i=1}^{n}(p, x_{i})(q,\omega_{i})=\evr(p \tn q )
\end{align*}
for all $p\in P$ and $q\in Q$, where $\cvl(1)= \sum_{i=1}^{n} \omega_{i}\otimes x_{i} $ and $\cvr(1)= \sum_{j=1}^{m} y_{j}\otimes \rho_{j}$  for positive integers $n,m$. The coproduct and counit of $H$, $\Delta$ and $\epsilon$, are defined by 
\begin{align*}
\Delta (a\ov{b})&= a\otimes \ov{b}, \quad \epsilon (a\ov{b})= ba \quad \epsilon((p,q))=\evl(p \tn q)\quad \epsilon((q,p))=\evr(q\tn p)
\\ \Delta &((q,p ))= \sum_{i=1}^{n} (q,\omega_{i} )\otimes (x_{i},p ) \quad \Delta ((p,q ))= \sum_{j=1}^{m} (p,y_{j} )\otimes (\rho_{j},q )  
\end{align*}
In \cite{ghobadi2020hopf}, we show that $H$ is not only a left Hopf algebroid in the sense of Schauenburg, but furthermore a Hopf algebroid in the sense of B{\"o}hm, and Szlach{\'a}nyi \cite{bohm2004hopf} and admits an invertible antipode $S$ acting as $S ((p,q)) = (q,p)$ with $S=S^{-1}$. 
\end{ex} 

\section{Generalizations}\label{SGenre}
\subsection{On Generalisation to Pivotal Diagrams}\label{SExtension}

By a \emph{pivotal diagram}, we mean a functor $\mathbb{D}: \ct{J}\rightarrow \ct{C}_{0}^{piv}$ from a small category $\ct{J}$ to the category $\ct{C}_{0}^{piv}$ which has pivotal pairs $(P,Q)$ as objects and morphism all $f:P_{1}\rightarrow P_{2}$ as morphism between $(P_{1},Q_{1})$ and $(P_{1},P_{2})$. Hence the datum for a pivotal diagram consists of sets of pivotal pairs $(P_{i},Q_{i})$ and pivotal morphisms $f_{j}:P_{j_{s}}\rightarrow P_{j_{t}}$ between them, where $i,j_{s},j_{t}\in I$ and $j\in J$ for index sets $I,J$. We define the category of $\mathbb{D}$-intertwined objects in $C$, denoted by $\ct{C}(\mathbb{D})$, as follows: the objects of $\ct{C}(\mathbb{D})$ are pairs $(X,\lbrace \sigma_{i}\rbrace_{i\in I} ) $, where $X$ is an object of $\ct{C}$ and $\lbrace \sigma_{i}\rbrace_{i\in I}$ a family of morphisms $\sigma_{i}: X\tn P_{i}\rightarrow  P_{i}\tn X$ for $i\in I$, so that for all $i\in I$, the pair $(X,\sigma_{i})$ belongs to $\ct{C}(P_{i},Q_{i})$, and for any $j\in J$, 
$$(f_{j}\tn X)\sigma_{j_{s}} = \sigma_{j_{t}} (X\tn f_{j}): X\tn P_{j_{s}} \rightarrow P_{j_{t}}\tn X $$ 
holds. Morphisms between objects $(X,\lbrace \sigma_{i}\rbrace_{i\in I} ) $ and $(Y,\lbrace \tau_{i}\rbrace_{i\in I} ) $ are morphisms $f:X\rightarrow Y$ in $\ct{C}$, which satisfy $\tau_{i} (f\tn P_{i})= (P_{i}\tn f) \sigma_{i}$ for all $i\in I$. 

Observe that $\ct{C}(\mathbb{D})$ lifts the monoidal structure of $\ct{C}$: we can define the tensor of two objects $(X,\lbrace \sigma_{i}\rbrace_{i\in I} ) $ and $(Y,\lbrace \tau_{i}\rbrace_{i\in I} ) $ of $\ct{C}(\mathbb{D})$ as 
$$(X,\lbrace \sigma_{i}\rbrace_{i\in I} ) \tn(Y,\lbrace \tau_{i}\rbrace_{i\in I} ) = \big( X\tn Y , \lbrace (\sigma_{i}\tn Y)(X\tn \tau_{i})\rbrace_{i\in I} \big) $$ 
By this definition, for any $i\in I$, the forgetful functor $U_{i}: \ct{C}(\mathbb{D})\rightarrow \ct{C}(P_{i},Q_{i})$ which sends a pair $(X,\lbrace \sigma_{i}\rbrace_{i\in I} ) $ to $(X,\sigma_{i})$, is strict monoidal. Consequently, the forgetful functor $U_{\mathbb{D}}: \ct{C}(\mathbb{D})\rightarrow \ct{C}$ sending pairs $(X,\lbrace \sigma_{i}\rbrace_{i\in I} ) $ to their underlying objects, $X$, also becomes strict monoidal. We must emphasize that the monoidal structure well-defined because for any $j\in J$, $f_{j}$ commutes with the relevant $P_{i}$-intertwinings, and thereby
\begin{align*}
(f_{j}\tn X\tn Y)(\sigma_{j_{s}}\tn Y)(X\tn \tau_{j_{s}}) &= (\sigma_{j_{t}}\tn Y) (X\tn f_{j}\tn Y)(X\tn \tau_{j_{s}})
 \\& =(\sigma_{j_{t}}\tn Y)(X\tn \tau_{j_{t}}) (X\tn f_{j})
\end{align*} 
holds. 
\begin{thm}\label{TCldDiag} If $\ct{C}$ is a left (right) closed monoidal category and $\mathbb{D}$ a pivotal diagram as described above, then $\ct{C}(\mathbb{D})$ has a left (right) closed monoidal structure which lifts that of $\ct{C}$ and the forgetful functor $U_{\mathbb{D}}$ is left (right) closed.
\end{thm} 
\begin{proof} In Theorem \ref{TCld}, we have provided suitable $P_{i}$-intertwinings for inner homs of two objects in $\ct{C}(P_{i},Q_{i})$ and demonstrated that the unit and counits of the adjunctions commute with these intertwinings. Hence, if $(X,\lbrace \sigma_{i}\rbrace_{i\in I} ) $ and $(Y,\lbrace \tau_{i}\rbrace_{i\in I} ) $ are objects of $\ct{C}(\mathbb{D})$  , we only need to check whether the induced $\mathbb{D}$-intertwinings $\lbrace\langle \sigma_{i},\tau_{i} \rangle_{l}\rbrace_{i\in I}$ and $\lbrace \langle \sigma_{i},\tau_{i} \rangle_{r}\rbrace_{i\in I}$ commute with morphisms $f_{j}$ so that $([X,Y]^{l},\lbrace \langle \sigma_{i},\tau_{i} \rangle_{l}\rbrace_{i\in I} ) $ and $([X,Y]^{r},\lbrace \langle \sigma_{i},\tau_{i} \rangle_{r}\rbrace_{i\in I} ) $ provide inner homs in $\ct{C}(\mathbb{D})$, and lift the closed structure of $\ct{C}$. Let $j\in J$, then 
\begin{align*}
\langle \sigma_{j_{t}},\tau_{j_{t}} \rangle (&[X,Y]^{l} f_{j}) = (P_{j_{t}}[A,(\evl_{j_{t}} B)(Q_{j_{t}}\tau_{j_{t}})(Q_{j_{t}}\epsilon^{A}_{B}P_{j_{t}})(Q_{j_{t}}[A,B]^{l}\sigma^{-1}_{j_{t}})]^{l})
\\& (P_{j_{t}}\eta^{A}_{Q_{j_{t}}[A,B]^{l}P_{j_{t}}})(\cvl_{j_{t}}[A,B]^{l}P_{j_{t}})([X,Y]^{l} f_{j})
\\= &(P_{j_{t}}[A,(\evl_{j_{t}} B)(Q_{j_{t}}\tau_{j_{t}})(Q_{j_{t}}Y f_{j}) (Q_{j_{t}}\epsilon^{A}_{B}P_{j_{s}})(Q_{j_{t}}[A,B]^{l}\sigma^{-1}_{j_{s}})]^{l})
\\ &(P_{j_{t}}\eta^{A}_{Q_{j_{t}}[A,B]^{l}P_{j_{s}}})(\cvl_{j_{t}}[A,B]^{l}P_{j_{s}})
\\ =& (P_{j_{t}}[A,(\evl_{j_{s}} B)(\pr{f_{j}}\tn \tau_{j_{s}}) (Q_{j_{t}}\epsilon^{A}_{B}P_{j_{s}})(Q_{j_{t}}[A,B]^{l}\sigma^{-1}_{j_{s}})]^{l})
\\ &(P_{j_{t}}\eta^{A}_{Q_{j_{t}}[A,B]^{l}P_{j_{s}}})(\cvl_{j_{t}}[A,B]^{l}P_{j_{s}})
\\ =& (P_{j_{t}}[A,(\evl_{j_{s}} B)(Q_{j_{s}} \tau_{j_{s}}) (Q_{j_{s}}\epsilon^{A}_{B}P_{j_{s}})(Q_{j_{s}}[A,B]^{l}\sigma^{-1}_{j_{s}})]^{l})
\\ &(P_{j_{t}}\eta^{A}_{Q_{j_{s}}[A,B]^{l}P_{j_{s}}})((P_{j_{t}} \pr{f_{j}}) \cvl_{j_{t}}[A,B]^{l}P_{j_{s}})
\\ = &(f_{j}[X,Y]^{l})(P_{j_{s}}[A,(\evl_{j_{s}} B)(Q_{j_{s}} \tau_{j_{s}}) (Q_{j_{s}}\epsilon^{A}_{B}P_{j_{s}})(Q_{j_{s}}[A,B]^{l}\sigma^{-1}_{j_{s}})]^{l})
\\ &(P_{j_{s}}\eta^{A}_{Q_{j_{s}}[A,B]^{l}P_{j_{s}}})( \cvl_{j_{s}}[A,B]^{l}P_{j_{s}})= (f_{j}[X,Y]^{l}) \langle \sigma_{j_{s}},\tau_{j_{s}} \rangle_{l}
\end{align*}
holds, where $\pr{f_{j}}:= (\evl_{j_{t}}Q_{j_{s}})(Q_{j_{t}} f_{j} Q_{j_{s}}) (Q_{j_{t}} \cvl_{j_{s}})$. A similar computation follows for $\lbrace \langle \sigma_{i},\tau_{i} \rangle_{r}\rbrace_{i\in I}$ and 
\begin{align*}
\langle \sigma_{j_{t}},\tau_{j_{t}} \rangle_{r} (&[X,Y]^{r} f_{j}) = (P_{j_{t}}[A,(\evl_{j_{t}} B)(Q_{j_{t}}\tau_{j_{t}})(Q_{j_{t}}\Theta^{A}_{B}P_{j_{t}})(\ov{\sigma_{j_{t}}}^{-1}[A,B]^{r}P_{j_{t}})]^{r})
\\& (P_{j_{t}}\Gamma^{A}_{Q_{j_{t}}[A,B]^{r}P_{j_{t}}})(\cvl_{j_{t}}[A,B]^{r}P_{j_{t}})([X,Y]^{r} f_{j})
\\ = &(P_{j_{t}}[A,(\evl_{j_{t}} B)(Q_{j_{t}}\tau_{j_{t}})(Q_{j_{t}}Bf_{j})(Q_{j_{t}}\Theta^{A}_{B}P_{j_{s}})(\ov{\sigma_{j_{t}}}^{-1}[A,B]^{r}P_{j_{s}})]^{r})
\\& (P_{j_{t}}\Gamma^{A}_{Q_{j_{t}}[A,B]^{r}P_{j_{s}}})(\cvl_{j_{t}}[A,B]^{r}P_{j_{s}})
\\ = &(P_{j_{t}}[A,(\evl_{j_{s}} B)(\pr{f_{j}}\tn \tau_{j_{s}})(Q_{j_{t}}\Theta^{A}_{B}P_{j_{s}})(\ov{\sigma_{j_{t}}}^{-1}[A,B]^{r}P_{j_{s}})]^{r})
\\& (P_{j_{t}}\Gamma^{A}_{Q_{j_{t}}[A,B]^{r}P_{j_{s}}})(\cvl_{j_{t}}[A,B]^{r}P_{j_{s}})
\\ = &(P_{j_{t}}[A,(\evl_{j_{s}} B)(Q_{j_{s}}\tn \tau_{j_{s}})(Q_{j_{s}}\Theta^{A}_{B}P_{j_{s}})(\ov{\sigma_{j_{s}}}^{-1}[A,B]^{r}P_{j_{s}})]^{r})
\\ &(P_{j_{t}}\Gamma^{A}_{Q_{j_{s}}[A,B]^{r}P_{j_{s}}})((P_{j_{t}} \pr{f_{j}}) \cvl_{j_{t}}[A,B]^{r}P_{j_{s}})
\\ = &(f_{j}[X,Y]^{r})(P_{j_{s}}[A,(\evl_{j_{s}} B)(Q_{j_{s}}\tn \tau_{j_{s}})(Q_{j_{s}}\Theta^{A}_{B}P_{j_{s}})(\ov{\sigma_{j_{s}}}^{-1}[A,B]^{r}P_{j_{s}})]^{r})
\\ &(P_{j_{s}}\Gamma^{A}_{Q_{j_{s}}[A,B]^{r}P_{j_{s}}})(\cvl_{j_{s}}[A,B]^{r}P_{j_{s}})= (f_{j}[X,Y]^{r}) \langle \sigma_{j_{s}},\tau_{j_{s}} \rangle_{r}
\end{align*}
holds. We are using the fact that $\ov{\sigma_{j_{t}}}^{-1}(\pr{f_{j}}A)= (A\pr{f_{j}} )\ov{\sigma_{j_{s}}}^{-1}  $ which follows from $(\pr{f_{j}}A)\ov{\sigma_{j_{s}}}=\ov{\sigma_{j_{t}}} (A\pr{f_{j}} )  $. 
\end{proof}
As in Section \ref{SMnd}, one can construct the relevant Hopf monad for a pivotal diagram, when suitable colimits exist. The monad will be a quotient of the coproduct of $T_{i}$, where $T_{i}$ are the respective Hopf monads of each pair $(P_{i},Q_{i})$. We would also like to point out that as mentioned in Remark \ref{RPiv}, if one considers the pivotal diagram $\mathbb{D}$, consisting of the object $(P,Q)$ and the morphism $\varrho_{P}:(P,Q)\rightarrow (P,Q)$ in a pivotal category $\ct{C}$, the obtained category $\ct{C}(\mathbb{D})$ will lift the pivotal structure of $\ct{C}$. We will make this statement more precise in the next section.
\subsection{$\ct{C}(P,Q)$ as a Dual of a Monoidal Functor}\label{SDual}
In this section we comment on the connection between our work and the recent work on Tannaka-Krein Duality for bimonads, as presented in \cite{shimizu2019tannaka}.

In \cite{majid1992braided}, Majid introduced the \emph{dual of a strong monoidal functor} $U:\ct{D}\rightarrow \ct{C}$ as the category whose objects are pairs $(X,\sigma : X\tn U \Rightarrow U\tn X )$, where $X$ is an object of $\ct{C}$ and $\sigma$ a natural monoidal isomorphism. One can recover the center of a monoidal category, $\ct{C}$, as the dual of the identity functor $\mathrm{id}_{\ct{C}}: \ct{C}\rightarrow \ct{C}$. The dual, denoted by $U^{\circ}$, has a natural monoidal structure, as in Relation \ref{EqMonoidal}, so that the relevant forgetful functor to $\ct{C}$ becomes strict monoidal. If we weaken this definition to allow the braidings, $\sigma$, to not be isomorphisms, we arrive at the concept of the \emph{lax dual} or \emph{weak center}, and must distinguish between the right and left dual, each having objects of the form $(X,\sigma : X\tn U \Rightarrow U\tn X )$ and $(X,\sigma : U\tn X \Rightarrow X\tn U )$, respectively. However, if $\ct{D}$ is rigid, then all lax braidings will be invertible and the lax left dual, the lax right dual and $U^{\circ}$, will all agree \cite{schauenburg2017dual}.  

In \cite{shimizu2019tannaka}, given a reconstruction data i.e. a strong monoidal functor $U: \ct{D}\rightarrow \ct{C}$, from an essentially small monoidal category $\ct{D}$, Shimizu constructs the bimonad whose modules would recover the left lax dual of $U$, and describes this monad as a coend. This generalizes the work of \cite{bruguieres2012quantum} on the centralizability of Hopf monads. The tools in \cite{shimizu2019tannaka} are then used to provide a monadic setting for the FRT reconstruction for a braided object in $\ct{C}$, where $\ct{D}$ is the category of braids. In the same vein, our work can be viewed as a reconstruction for a pivotal pair, where $\ct{D}$ is the simplest pivotal category 

A choice of a pivotal pair in a monoidal category $\ct{C}$, exactly corresponds to a strict monoidal functor from the monoidal category generated by a single pivotal object, which we denote by $\mathrm{Piv}(1)$. In the language of \emph{monoidal signatures}, as described in Section 6.1 of \cite{shimizu2019tannaka}, the category $\mathrm{Piv}(1)$ is the monoidal category generated by two objects $1$ and $2$ and two pairs of duality morphisms making $2$ both the left and right dual of $1$. From this point of view, it should be clear that given any pivotal pair in $\ct{C}$, we have a strict monoidal functor $(P,Q):\mathrm{Piv}(1)\rightarrow \ct{C}$ which sends $1$ to $P$ and $2$ to $Q$. Additionally, any pivotal morphism $f:(P_{1},Q_{1})\rightarrow (P_{2},Q_{2})$ provides a natural transformation between the corresponding functors $\mathbf{f}: (P_{1},Q_{1})\Rightarrow (P_{2},Q_{1})$ with $\mathbf{f}_{1}=f$ and $\mathbf{f}_{2}= (\evl_{2} Q_{1} ) ( Q_{2}  f  Q_{1} ) (Q_{2} \cvl_{1})$.

Provided with a pivotal pair $(P,Q)$, the category $\ct{C}(P,Q)$ is exactly the dual of the corresponding monoidal functor $(P,Q)$. Any object $(A,\sigma)$ of $\ct{C}(P,Q)$ has a natural braiding $\mathbf{\sigma}$ defined by $\mathbf{\sigma}_{P}=\sigma$ and $\mathbf{\sigma}_{Q}=\ov{\sigma}^{-1}$ and it follows easily that the duality morphisms commute with these braidings. Conversely, for any object $(A,\mathbf{\sigma})$ in $(P,Q)^{\circ}$, $(A,\mathbf{\sigma}_{P})$ naturally becomes an object in $\ct{C}(P,Q)$, since the duality morphisms commute with $\mathbf{\sigma}$, which implies that $\mathbb{\sigma}_{Q}$ and $\ov{\sigma}$ must be inverses. We should also note that in \cite{schauenburg2017dual}, it is shown that the dual of a monoidal functor from a rigid category, lifts the left inner homs of the target category, when they exist. This result provides an indirect proof for Theorem \ref{TCld}. Additionally, given the reconstruction data $(P,Q):\mathrm{Piv}(1)\rightarrow \ct{C}$, the monad constructed in \cite{shimizu2019tannaka} is given by the coend 
$$ T(X)=\int^{a\in \mathrm{Piv}(1)}  (P,Q)(a) \tn X \tn (P,Q)(a)^{\vee}$$ 
which would exactly reduce to our definition in Section \ref{SMnd}. From this point of view, the parallel pairs described in Section \ref{SMnd} are the only pairs we consider since $\mathrm{Piv}(1)$ is only generated by the four duality morphisms. Similarly, a pivotal diagram as in Section \ref{SExtension} can be viewed as a strict monoidal functor from the category generated by several pivotal pairs and morphisms between them, with the relevant composition relations. Consequently, $\ct{C}(\mathbb{D})$ can be viewed as the dual of this functor, which by Theorem \ref{TCldDiag} will again be closed. 

We conclude this section by the following extension of Theorem \ref{TpivCP}: 
\begin{thm}\label{TDualPiv} Assume $U:\ct{D}\rightarrow \ct{C}$ is a pivotal functor between pivotal categories and strictly preserves duality morphisms i.e. the induced isomorphisms $\zeta : F(\pr{-})\rightarrow \pr{F(-)}$ are the identity morphisms. Then the dual of $U$ is a pivotal category such that the forgetful functor from $U^{\circ}$ to $\ct{C}$ preserves its pivotal structure.
\end{thm}
\begin{proof} Since $\ct{C}$ is rigid, $U^{\circ}$ is a rigid category and the braidings on the duals are given exactly as described in Corollary \ref{CRig}. Hence the proof follows exactly as the proof of Theorem \ref{TpivCP}, where $\sigma$ is now a natural transformation. In the calculations of Theorem \ref{TpivCP}, we required $\ov{\sigma}$ to commute with $\varrho_{Q}$ i.e. in this general setting we need the pivotal structure $\varrho^{\ct{C}}$ in $\ct{C}$ to commute with the braidings. However, this follows directly by our assumptions in the statement and that the braidings commute with $U(\varrho^{\ct{D}})$, where $\varrho^{\ct{D}}$ denotes the pivotal structure of $\ct{D}$. 
\end{proof}

\section{Diagrams and Proof of Theorem 4.2}\label{SDiag}
In this section, we provide the necessary diagrams for the proof of Theorem \ref{TCld}. The large commutative rectangles marked with $\spadesuit$, commute by the naturality of the unit and the triangle identities. 
\vspace{3cm}

\adjustbox{scale=0.95,center}{
\begin{tikzcd}[row sep=3.3em ]\label{DiagUnit}
BP \arrow[d, "\eta_{B}^{A}P"'] \arrow[rd, "\cvl BP"] \arrow[rr, "\sigma_{B}"] &                                                                                                                                                              & P B \arrow[d, "P\eta_{B}^{A}"]                                             \\
{{[A,B A]^{l}}  P} \arrow[d, "{\cvl {[A,B A]^{l}  P }}"']                                       & PQBP \arrow[d] \arrow[ld, "PQ\eta^{A}_{B}P"'] \arrow[ru, "P(B\evl)(\ov{\sigma_{B}}P)"] \arrow[d, "P\eta^{A}_{QBP}"']                                          & {P [A,B A]^{l} }                                                           \\
{PQ[A,B A]^{l}  P } \arrow[d, "{{P\eta^{A}_{Q[A,B A]^{l}  P}} }"']                              & {P[A,QB  PA]^{l}} \arrow[d, "{P[A,QB  \sigma_{A}^{-1}]^{l}}"] \arrow[ld, "{P[A,Q\eta_{B}^{A}  PA]^{l}}" description] \arrow[ru, "{P[A,(B\evl)(\ov{\sigma_{B}}P)A]^{l}}" description] & {P[A,Q PBA]^{l} } \arrow[u, "{P[A,\evl BA]^{l} }"']                         \\
{P[A,Q[A,B A]^{l}  PA]^{l}} \arrow[d, "{ {P[A,Q[A,B A]^{l}  \sigma_{A}^{-1}]^{l} }}" description]          & {P[A,QBA  P]^{l} } \arrow[ld, "{ {P[A,Q\eta_{B}^{A} AP]^{l}} }"] \arrow[rd, equal]                                                                                 &                                                                            \\
{P[A,Q[A,B A]^{l}  AP]^{l} } \arrow[rr, "{ {P[A,Q\epsilon^{A}_{BA}  P]^{l} }}"']                &                                                                                                                                                              & {P[A,QBA  P]^{l} } \arrow[uu, "{P[A,Q(\sigma_{B}\tn \sigma_{A})]^{l} }"' description]
\end{tikzcd}}\captionof{figure}{Proof of unit commuting with $P$-intertwinings}

\vspace{0.5cm}\adjustbox{scale=0.95,center}{
\begin{tikzcd}[row sep=3.3em]\label{DiagCounit}
{[A,B]^{l}AP} \arrow[r, "{[A,B]^{l}\sigma_{A}}"] \arrow[d, equal]           & {[A,B]^{l}PA} \arrow[ld, "{[A,B]^{l}\sigma_{A}^{-1}}" description] \arrow[rr, "{\cvl [A,B]^{l}PA}"] &  & {PQ[A,B]^{l}PA} \arrow[d, "{{P\eta^{A}_{Q[A,B]^{l}P}A} }"] \arrow[lld, equal]                                                  \\
{[A,B]^{l}AP} \arrow[rd, "{\cvl[A,B]^{l}AP}" description] \arrow[d, "\epsilon_{B}^{A}P"'] & {PQ[A,B]^{l}PA} \arrow[d, "{{PQ[A,B]^{l}\sigma^{-1}_{A}}  }"]                                       &  & {P[A,Q[A,B]^{l}PA]^{l}A} \arrow[dd, "{{P[A,Q[A,B]^{l}\sigma_{A}^{-1}]^{l}A} }" description] \arrow[ll, "{{P\epsilon_{Q[A,B]^{l}PA}^{A}} }"'] \\
BP \arrow[rd, "\cvl BP" description] \arrow[d, "\sigma_{B}"']                             & {PQ[A,B]^{l}AP} \arrow[d, "{PQ\epsilon_{B}^{A}P} "]                                                 &  &                                                                                                                                              \\
PB                                                                                        & PQBP \arrow[l, "P(\evl Q)(Q\sigma_{B})"]                                                            &  & {P[A,Q[A,B]^{l}AP]^{l}A} \arrow[d, "{{P[A,Q\epsilon_{B}^{A}P]^{l}A} }"] \arrow[llu, "{{P\epsilon^{A}_{Q[A,B]^{l}AP}} }" description]         \\
{P[A,B]^{l}A} \arrow[u, "P\epsilon_{B}^{A}"]                                              &                                                                                                     &  & {P[A,QBP]^{l}A} \arrow[lll, "{P[A,(\evl B)(Q\sigma_{B})]^{l}A}"] \arrow[llu, "P\epsilon^{A}_{QBP}"]                                         
\end{tikzcd}}\captionof{figure}{Proof of counit commuting with $P$-intertwinings}
\thispagestyle{empty}

\begin{landscape}
\adjustbox{scale=0.7,center}{

\begin{tikzcd}[row sep=4.5em,column sep=4em, cramped]\label{Diag[AB]P}
                                                                                                                                &                                                                                                                              & {{[A,B]^{l}}P} \arrow[rrd, "{{\eta^{A}_{[A,B]^{l}}}P}"] \arrow[rrrd, equal, bend left]                                                                    &                                                                                        &                                                                 &                                                                                   \\
                                                                                                                                & {{[A,B]^{l}}PQP} \arrow[d, "{{\cvl[A,B]^{l}}PQP}"] \arrow[r, "{\eta^{A}_{[A,B]^{l}PQ}P}"] \arrow[ru, "{{[A,B]^{l}}\evr P}"] & {{[A,[A,B]^{l}PQA]^{l}}P} \arrow[d, "{{[A,\cvl [A,B]PQA]^{l}}P}"description] \arrow[rd, "{{[A,[A,B]^{l}P\ov{\sigma_{A}}^{-1}]^{l}}P}"] \arrow[rr, "{{[A,[A,B]^{l}\evr A]^{l}}P}"] &                                                                                        & {[A,[A,B]^{l}A]^{l}P} \arrow[r, "{[A,\epsilon_{B}^{A}]^{l}}P"]                           & {{[A,B]^{l}}P}                                                                    \\
{{[A,B]^{l}}P} \arrow[d, "{{\cvl [A,B]^{l}}P }"'] \arrow[ru, "{{[A,B]^{l}}P\cvr}"description] \arrow[rruu, equal, bend left] & {PQ{[A,B]^{l}}PQP} \arrow[d, "{P{\eta^{A}_{Q[A,B]^{l}P}}QP }"]                                                              & {{[A,PQ[A,B]^{l}PQA]^{l}}P} \arrow[d, "{{[A,PQ[A,B]^{l}P\ov{\sigma_{A}}^{-1}]^{l}}P}"description]                                                                                 & {{[A,[A,B]^{l}PAQ]^{l}}P} \arrow[ld, "{{[A,\cvl [A,B]^{l}PAQ]^{l}}P}"] \arrow[r, "{{[A,[A,B]^{l}\sigma^{-1}_{A}Q]^{l}}P}",bend left] & {[A,[A,B]^{l}APQ]^{l}P}\arrow[ldd, "{[A,\cvl [A,B]^{l}APQ]^{l}P}"'] \arrow[d, "{[A,\epsilon_{B}^{A}PQ]^{l}P}"] \arrow[u, "{[A,[A,B]^{l}A\evr ]^{l}P}"']          & {[A,BPQ]^{l}P} \arrow[u, "{[A,B\evl ]^{l}P}"']                                                       
    \\
{PQ{[A,B]^{l}}P} \arrow[d, "{P{\eta^{A}_{Q[A,B]^{l}P}}}"'] \arrow[ru, "{PQ{[A,B]^{l}}P\cvr}"description]                                   & {P{[A,Q[A,B]^{l}PA]^{l}}QP} \arrow[d, "{P{[A,Q[A,B]^{l}\sigma^{-1}_{A}]^{l}}QP }"]                                          & {{[A,PQ[A,B]^{l}PAQ]^{l}}P} \arrow[rd, "{ {[A,PQ[A,B]^{l}\sigma_{A}^{-1}Q]^{l}}P}"']                                                                                    &                                                                                        & {[A,BPQ]^{l}P} \arrow[dd, "{[A,\cvl BPQ ]^{l}P}"', bend right=18] \arrow[ru, equal, bend right] & {[A,PBQ]^{l}P} \arrow[u, "{[A,\sigma^{-1}_{B}Q]^{l}P}"']                                                               \\
{P[A,Q[A,B]^{l}PA]^{l}} \arrow[d, "{P[A,Q[A,B]^{l}\sigma_{A}^{-1}]^{l}}"'] \arrow[ru, "{{P[A,Q[A,B]^{l}PA]^{l}}\cvr}"description]          & {P[A,Q[A,B]^{l}AP]^{l}QP}\arrow[rr, "\spadesuit", phantom] \arrow[rd, "{P[A,Q\epsilon_{B}^{A}P]^{l}QP}"']                                                                                                         &                                                                                                                                                                         & {{[A,PQ[A,B]^{l}APQ]^{l}}P} \arrow[rd, "{{[A,PQ\epsilon^{A}_{B}PQ]^{l}}P }"']          &                                                                 & {{[A,P[A,B]^{l}AQ]^{l}}P} \arrow[u, "{{[A,P\epsilon_{B}^{A}Q]^{l}}P}"']                                         \\
{P[A,Q[A,B]^{l}AP]^{l}} \arrow[rd, "{P[A,Q\epsilon_{B}^{A}P]^{l}} "'] \arrow[ru, "{{P[A,Q[A,B]^{l}AP]^{l}}\cvr}"description]                                            &                                                                                                                              & {P[A,QBP]^{l}QP} \arrow[r, "{P[A,Q\sigma_{B}]^{l}QP}"']                                                                                                                                       & {P[A,QPB]^{l}QP} \arrow[rd, "{P[A,\evl B]^{l}QP}"']                                                     & {{[A,PQBPQ]^{l}}P} \arrow[ruu, "{{[A,P(\evl B)(Q\sigma_{B})Q]^{l}}P}"description, bend left]                           & {{[A,P[A,B]^{l}QA]^{l}}P} \arrow[u, "{{[AP[A,B]^{l}\ov{\sigma_{A}}^{-1}]^{l}}P}"'] \\
                                                                                                                                & {P[A,QBP]^{l}} \arrow[r, "{P[A,Q\sigma_{B}]^{l}}"'] \arrow[ru, "{{P[A,QBP]^{l}}\cvr}"]                                       & {P[A,QPB]^{l}} \arrow[r, "{P[A,\evl B]^{l}}"'] \arrow[ru, "{{P[A,QPB]^{l}} \cvr}"]                                                                                      & {P[A,B]^{l}} \arrow[r, "{P[A,B]^{l}\cvr}"']                                            & {{P[A,B]^{l}}QP} \arrow[ru, "{\eta^{A}_{P[A,B]^{l}Q}P} "']                               &                                                                                  
\end{tikzcd}
}\captionof{figure}{Proof of $ \langle \sigma_{A},\sigma_{B} \rangle_{l}^{-1} \langle \sigma_{A},\sigma_{B} \rangle_{l}= \mathrm{id}_{[A,B]^{l}P}$}
\thispagestyle{empty}
\end{landscape}

\begin{landscape}
\adjustbox{scale=0.7,center}{

\begin{tikzcd}[row sep=4.5em,column sep=5em, cramped]\label{DiagP[AB]}
                                                                                                                                                                                                                   & {PQP[A,B]^{l}} \arrow[rrrd, "{P\eta^{A}_{QP[A,B]^{l}}}"'] \arrow[r, "{P\evl [A,B]^{l}}"] \arrow[d, "{PQP[A,B]^{l}\cvr}"] & {P[A,B]^{l}} \arrow[rr, "{P\eta^{A}_{[A,B]^{l}}}"] \arrow[rrr, equal, bend left] &                                                                         & {P[A,[A,B]^{l}A]^{l}} \arrow[r, "{P[A,\epsilon_{B}^{A}]^{l}}"']                                                                                                                                   & {P[A,B]^{l}}                                                               \\
{P[A,B]^{l}} \arrow[d, "{{P[A,B]^{l}\cvr} }"'] \arrow[ru, "{{\cvl P[A,B]^{l}} }"description] \arrow[rru, equal, bend left=60] & {PQP[A,B]^{l}QP} \arrow[d, "{PQ\eta^{A}_{P[A,B]^{l}Q}P}"] \arrow[r, "{P\eta^{A}_{QP[A,B]^{l}QP}}"']                      & {P[A,QP[A,B]^{l}QPA]^{l}} \arrow[d, "{P[A,QP[A,B]^{l}Q\sigma_{A}^{-1}]^{l}}"]                  &                                                                         & {P[A,QP[A,B]^{l}A]^{l}} \arrow[ll, "{P[A,QP[A,B]^{l}\cvr A]^{l}}"] \arrow[ldd, "{P[A,QP[A,B]^{l}A\cvr ]^{l}}"'] \arrow[d, "{P[A,QP\epsilon^{A}_{B}]^{l}}"] \arrow[u, "{P[A,\evl [A,B]^{l}A]^{l}}"] & {P[A,QPB]^{l}} \arrow[u, "{P[A,\evl B]^{l}}"']                              \\
{P[A,B]^{l}QP} \arrow[d, "{{\eta^{A}_{P[A,B]^{l}Q}P} }"'] \arrow[ru, "{\cvl P[A,B]^{l}QP}"description]                                      & {PQ[A,P[A,B]^{l}QA]^{l}P} \arrow[d, "{PQ[A,P[A,B]^{l}\ov{\sigma_{A}}]^{l}P}"]                                             & {P[A,QP[A,B]^{l}QAP]^{l}} \arrow[rd, "{P[A,QP[A,B]^{l}\ov{\sigma_{A}}P]^{l}}"']                 &                                                                         & {P[A,QPB]^{l}} \arrow[r, "{P[A,Q\sigma_{B}^{-1}]^{l}}"] \arrow[dd, "{P[A,QPB\cvr]^{l}}"', bend right=15] \arrow[ru, equal]                                                                          & {P[A,QBP]^{l}} \arrow[u, "{{P[A,Q\sigma_{B}]^{l}} }"']                     \\
{[A,P[A,B]^{l}QA]^{l}P } \arrow[d, "{[A,P[A,B]^{l}\ov{\sigma_{A}}]^{l}P }"'] \arrow[ru, "{\cvl [A,P[A,B]^{l}QA]^{l}P }"description]         & {PQ[A,P[A,B]^{l}AQ]^{l}P }\arrow[rr, "\spadesuit", phantom] \arrow[rd, "{[A,P\epsilon_{B}^{A}Q]^{l}P }"']                                                  &                                                                                                 & {P[A,QP[A,B]^{l}AQP]^{l}} \arrow[rd, "{P[A,QP\epsilon^{A}_{B}QP]^{l}}"'] &                                                                                                                                                                                                   & {P[A,Q[A,B]^{l}AP]^{l}} \arrow[u, "{P[A,Q\epsilon_{B}^{A}P]^{l}}"']        \\
{[A,P[A,B]^{l}AQ]^{l}P } \arrow[rd, "{[A,P\epsilon_{B}^{A}Q]^{l}P }"'] \arrow[ru, "{\cvl[A,PBQ]^{l}P }"description]                         &                                                                                                                           & {PQ[A,PBQ]^{l}P } \arrow[r, "{[A,\sigma^{-1}_{B}Q]^{l}P }"']                                    & {PQ[A,BPQ]^{l}P } \arrow[rd, "{{PQ[A,B\evr ]^{l}P } }"']                & {P[A,QPBQP]^{l}} \arrow[ruu, "{P[A,Q(B\evr)(\sigma^{-1}_{B}Q)P]^{l}}"description, bend left]                                                                                                                           & {P[A,Q[A,B]^{l}PA]^{l}} \arrow[u, "{P[A,Q[A,B]^{l}\sigma^{-1}_{A}]^{l}}"'] \\
                                                                                                                                 & {[A,PBQ]^{l}P } \arrow[r, "{[A,\sigma^{-1}_{B}Q]^{l}P }"'] \arrow[ru, "{\cvl [A,PBQ]^{l}P }"]                             & {[A,BPQ]^{l}P } \arrow[r, "{[A,B\evr]^{l}P }"'] \arrow[ru, "{\cvl[A,BPQ]^{l}P }"]               & {[A,B]^{l}P } \arrow[r, "{\cvl [A,B]^{l}P }"']                          & {PQ[A,B]^{l}P } \arrow[ru, "{P\eta^{A}_{Q[A,B]^{l}P }}"']                                                                                                                                         &                                                                           
\end{tikzcd}
}\captionof{figure}{Proof of $ \langle \sigma_{A},\sigma_{B} \rangle_{l} \langle \sigma_{A},\sigma_{B} \rangle_{l}^{-1}= \mathrm{id}_{P[A,B]^{l}}$}
\thispagestyle{empty}
\end{landscape}

\begin{landscape}
\adjustbox{scale=0.7,center}{
\begin{tikzcd}[row sep=4.5em,column sep=5em, cramped]\label{DiagQ[AB]}
                                                                                                                                   & {QPQ[A,B]^{l}} \arrow[d, "{QPQ[A,B]^{l}\cvl }"description] \arrow[rd, "{Q\eta^{A}_{PQ[A,B]^{l}}}"] \arrow[r, "{Q\evr [A,B]^{l}}"'] & {Q[A,B]^{l}} \arrow[rr, "{Q\eta^{A}_{[A,B]^{l}}}"] \arrow[rrr,equal, bend left]                                                                                                                           &                                                                          & {Q[A,[A,B]^{l}A]^{l}} \arrow[r, "{Q[A,\epsilon_{B}^{A}]^{l}}"']                                                           & {Q[A,B]^{l}}                                                               \\
{Q[A,B]^{l}} \arrow[d, "{Q[A,B]^{l}\cvl}"'] \arrow[ru, "{\cvr Q[A,B]^{l}}"description] \arrow[rru,equal, bend left=60]         & {QPQ[A,B]^{l}PQ} \arrow[d, "{QP\eta^{A}_{Q[A,B]^{l}P}Q}"description] \arrow[rd, "{Q\eta^{A}_{PQ[A,B]^{l}PQ}}"description]                     & {Q[A,PQ[A,B]^{l}A]^{l}} \arrow[d, "{Q[A,PQ[A,B]^{l}\cvl A]^{l}}"description] \arrow[rd, "{Q[A,PQ[A,B]^{l}A\cvl ]^{l}}"] \arrow[rr, "{Q[A,PQ\epsilon^{A}_{B}]^{l}}"'] \arrow[rru, "{Q[A,\evr [A,B]^{l}A]^{l}}"] &                                                                          & {Q[A,PQB]^{l}} \arrow[dd, "{Q[A,PQB\cvl]^{l}}"] \arrow[rd, "{Q[A,P\ov{\sigma_{B}}]^{l}}"] \arrow[ru, "{Q[A,\evr B]^{l}}"] & {Q[A,BPQ]^{l}} \arrow[u, "{Q[A,B\evr ]^{l}}"']                             \\
{Q[A,B]^{l}PQ} \arrow[ru, "{\cvr Q[A,B]^{l}PQ}"description] \arrow[d, "{\eta^{A}_{Q[A,B]^{l}P}Q}"']                                           & {QP[A,Q[A,B]^{l}PA]^{l}Q} \arrow[d, "{QP[A,Q[A,B]^{l}\sigma^{-1}_{A}]^{l}Q}"]                                            & {Q[A,PQ[A,B]^{l}PQA]^{l}} \arrow[d, "{Q[A,PQ[A,B]^{l}P\ov{\sigma_{A}}]^{l}}"description]                                                                                                                      & {Q[A,PQ[A,B]^{l}APQ]^{l}} \arrow[rd, "{Q[A,PQ\epsilon_{B}^{A}PQ]^{l}}"] &                                                                                                                           & {Q[A,PBQ]^{l}} \arrow[u, "{Q[A,\sigma_{B}^{-1}Q]^{l}}"']                   \\
{[A,Q[A,B]^{l}PA]^{l}Q} \arrow[d, "{[A,Q[A,B]^{l}\sigma^{-1}_{A}]^{l}Q}"'] \arrow[ru, "{\cvr [A,Q[A,B]^{l}\sigma^{-1}_{A}]^{l}Q}"description] & {QP[A,Q[A,B]^{l}AP]^{l}Q} \arrow[rd, "{{QP[A,Q\epsilon^{A}_{B}P]^{l}Q} }"']                                               & {Q[A,PQ[A,B]^{l}PAQ]^{l}}  \arrow[ru, "{Q[A,PQ[A,B]^{l}\sigma^{-1}_{A}Q]^{l}}"']                                                                                        &                                                                          & {Q[A,PQBPQ]^{l}} \arrow[ru, "{Q[A,P(\evl B)(Q\sigma_{B})Q]^{l}}"description]                                                        & {Q[A,P[A,B]^{l}AQ]^{l}} \arrow[u, "{Q[A,P\epsilon_{B}^{A}Q]^{l}}"']        \\
{[A,Q[A,B]^{l}AP]^{l}Q} \arrow[rd, "{[A,Q\epsilon_{B}^{A}P]^{l}Q}"'] \arrow[ru, "{\cvr [A,Q[A,B]^{l}AP]^{l}Q}"description]                   &                                                                                                                           & {QP[A,QBP]^{l}Q} \arrow[r, "{QP[A,Q\sigma_{B}]^{l}Q}"']                                                                                                                                             & {QP[A,QPB]^{l}Q} \arrow[rd, "{QP[A,\evl B]^{l}Q}"] \arrow[ru, "\spadesuit", phantom]                      &                                                                                                                           & {Q[A,P[A,B]^{l}QA]^{l}} \arrow[u, "{Q[A,P[A,B]^{l}\ov{\sigma_{A}}]^{l}}"'] \\
                                                                                                                                   & {[A,QBP]^{l}Q} \arrow[r, "{[A,Q\sigma_{B}]^{l}Q}"'] \arrow[ru, "{\cvr[A,QBP]^{l}Q}"']                                     & {[A,QPB]^{l}Q} \arrow[r, "{[A,\evl B]^{l}Q}"'] \arrow[ru, "{\cvr[A,QPB]^{l}Q}"']                                                                                                                    & {[A,B]^{l}Q} \arrow[r, "{\cvr [A,B]^{l}Q}"']                             & {QP[A,B]^{l}Q} \arrow[ru, "{Q\eta^{A}_{P[A,B]^{l}Q}}"']                                                                   &                                                                           
\end{tikzcd}
}\captionof{figure}{Proof of $ \ov{\langle \sigma_{A},\sigma_{B} \rangle_{l}}^{-1}\ov{ \langle \sigma_{A},\sigma_{B} \rangle_{l}}= \mathrm{id}_{Q[A,B]^{l}}$}
\thispagestyle{empty}
\end{landscape}
\begin{landscape}
\adjustbox{scale=0.7,center}{
\begin{tikzcd}[row sep=4.5em,column sep=5em, cramped]\label{Diag[AB]Q}
                                                                                                                                                                                                               &                                                                                                                          & {[A,B]^{l}Q} \arrow[rrd, "{\eta^{A}_{[A,B]^{l}}Q}"] \arrow[rrrd,equal,  bend left]                                                                                   &                                                                                                                                                            &                                                                                                                                 &                                                                            \\
                                                                                                                                & {[A,B]^{l}QPQ} \arrow[d, "{\cvr[A,B]^{l}QPQ}"description] \arrow[r, "{\eta_{[A,B]^{l}QP}^{A}Q}"'] \arrow[ru, "{[A,B]^{l}\evl Q}"'] & {[A,[A,B]^{l}QPA]^{l}Q} \arrow[d, "{[A,\cvr[A,B]^{l}QPA]^{l}Q}"'] \arrow[rd, "{[A,[A,B]^{l}Q\sigma^{-1}_{A}]^{l}Q}"] \arrow[rr, "{[A,[A,B]^{l}\evl A]^{l}Q}"] &                                                                                                                                                            & {[A,[A,B]^{l}A]^{l}Q} \arrow[r, "{[A,\epsilon_{B}^{A}]^{l}Q}"]                                                                       & {[A,B]^{l}Q}                                                               \\
{[A,B]^{l}Q} \arrow[d, "{{\cvr[A,B]^{l}Q}  }"'] \arrow[ru, "{{[A,B]^{l}Q\cvl}  }"description] \arrow[rruu,equal, bend left] & {QP[A,B]^{l}QPQ} \arrow[d, "{{Q\eta^{A}_{P[A,B]^{l}Q}} PQ}"] \arrow[r, "{\eta_{QP[A,B]^{l}QP}^{A}Q}"']                  & {[A,QP[A,B]^{l}QPA]^{l}Q} \arrow[d, "{[A,QP[A,B]^{l}Q\sigma^{-1}_{A}]^{l}Q}"']                                                                                & {[A,[A,B]^{l}QAP]^{l}Q} \arrow[ld, "{[A,\cvr [A,B]^{l}QAP]^{l}Q}", bend right] \arrow[d, "{[A,[A,B]^{l}\ov{\sigma_{A}}P]^{l}Q}"description]                                      &                                                                                                                                 & {[A,QPB]^{l}Q} \arrow[u, "{[A,\evl B]^{l}Q}"']                             \\
{QP[A,B]^{l}Q} \arrow[d, "{{Q\eta^{A}_{P[A,B]^{l}Q}} }"'] \arrow[ru, "{{QP[A,B]^{l}Q\cvl} }"description]                                   & {Q[A,P[A,B]^{l}QA]^{l}PQ} \arrow[d, "{Q[A,P[A,B]^{l}\ov{\sigma_{A}}]^{l}PQ}"]                                           & {[A,QP[A,B]^{l}QAP]^{l}Q} \arrow[d, "{[A,QP[A,B]^{l}\ov{\sigma_{B}}P]^{l}Q}"description]                                                                                & {[A,[A,B]^{l}AQP]^{l}Q} \arrow[ld, "{[A,\cvr [A,B]^{l}AQP]^{l}Q}"] \arrow[r, "{[A,\epsilon_{B}^{A}QP]^{l}Q}"'] \arrow[ruu, "{[A,[A,B]^{l}A\evl ]^{l}Q}"'] & {[A,BQP]^{l}Q} \arrow[d, "{[A,\cvr BQP]^{l}Q}"'] \arrow[ruu, "{[A,B\evl ]^{l}Q}"'] \arrow[r, "{[A,\ov{\sigma_{B}}^{-1}P]^{l}Q}"] & {[A,QBP]^{l}Q} \arrow[u, "{[A,Q\sigma_{A}]^{l}Q}"']                        \\
{Q[A,P[A,B]^{l}QA]^{l}} \arrow[d, "{Q[A,P[A,B]^{l}\ov{\sigma_{A}}]^{l}}"'] \arrow[ru, "{{Q[A,P[A,B]^{l}QA]^{l}}\cvl}"description]          & {{Q[A,P[A,B]^{l}AQ]^{l}}PQ} \arrow[rd, "{{Q[A,P\epsilon_{B}^{A}Q]^{l}}PQ}"']                                             & {[A,QP[A,B]^{l}AQP]^{l}Q} \arrow[rr, "{[A,QP\epsilon_{B}^{A}QP]^{l}Q}"']                                                                                      &                                                                                                                                                            & {[A,QPBQP]^{l}Q} \arrow[ru, "{[A,Q(B\evr)(\sigma_{B}^{-1}Q)P]^{l}Q}"description]                                                                                              & {[A,Q[A,B]^{l}AP]^{l}Q} \arrow[u, "{[A,Q\epsilon^{A}_{B}P]^{l}Q}"']        \\
{Q[A,P[A,B]^{l}AQ]^{l}} \arrow[rd, "{Q[A,P\epsilon_{B}^{A}Q]^{l}}"'] \arrow[ru, "{Q[A,P[A,B]^{l}AQ]^{l}\cvl}"description]                  &                                                                                                                          & {Q[A,PBQ]^{l}PQ} \arrow[r, "{Q[A,\sigma_{B}^{-1}Q]^{l}PQ}"]                                                                                                  & {Q[A,BPQ]^{l}PQ} \arrow[rd, "{Q[A,B\evr]^{l}PQ}"']                                                                                                         &                                                                                                                                 & {[A,Q[A,B]^{l}PA]^{l}Q} \arrow[u, "{[A,Q[A,B]^{l}\sigma_{A}^{-1}]^{l}Q}"'] \\
                                                                                                                                & {Q[A,PBQ]^{l}} \arrow[r, "{Q[A,\sigma_{B}^{-1}Q]^{l}}"'] \arrow[ru, "{Q[A,PBQ]^{l}\cvl}"]                                & {Q[A,BPQ]^{l}} \arrow[r, "{Q[A,B\evr]^{l}}"'] \arrow[ru, "{Q[A,BPQ]^{l}\cvl}"]                                                                                & {Q[A,B]^{l}} \arrow[r, "{Q[A,B]^{l}\cvl}"']                                                                                                                & {Q[A,B]^{l}PQ} \arrow[uu, "\spadesuit", phantom]\arrow[ru, "{\eta^{A}_{Q[A,B]^{l}P}Q}"']                                                                         &                                                                           
\end{tikzcd}
}\captionof{figure}{Proof of $ \ov{\langle \sigma_{A},\sigma_{B} \rangle_{l}}\ov{ \langle \sigma_{A},\sigma_{B} \rangle_{l}}^{-1}= \mathrm{id}_{[A,B]^{l}Q}$}
\thispagestyle{empty}
\end{landscape}
\bibliographystyle{plain}
\bibliography{hopf}
\end{document}